\documentclass[10pt,twoside,a4paper]{article}%
\usepackage{amsfonts}
\usepackage{amssymb}
\usepackage{amsmath}
\usepackage{amsthm}
\usepackage{graphicx}
\usepackage[latin1]{inputenc}
\usepackage{color}
\usepackage{bm}%
\providecommand{\U}[1]{\protect\rule{.1in}{.1in}}
\newtheorem{theorem}{Theorem}[section]

\newtheorem{corollary}[theorem]{Corollary}

\newtheorem{example}[theorem]{Example}

\newtheorem{lemma}[theorem]{Lemma}

\newtheorem{remark}[theorem]{Remark}

\begin{document}

\title{Optimal control of forward-backward mean-field stochastic delayed systems}
\author{\ Nacira AGRAM\thanks{Faculty of economic sciences and management, University
Med Khider, Po. Box 145, Biskra $\left(  07000\right)  $ Algeria. Email:
agramnacira@yahoo.fr} \ and Elin Engen Røse\thanks{Department of Mathematics,
University of Oslo, Box 1053 Blindern, N-0316 Oslo, Norway. Email:
elinero@math.uio.no}\ \ \ \ \
\and \textit{Dedicated to Professor Bernt Øksendal on the occasion of his 70th
birthday}}
\date{\ 16 December 2014}
\maketitle

\begin{abstract}
We study methods for solving stochastic control problems of systems of
forward-backward mean-field equations with delay, in finite or infinite
horizon. Necessary and sufficient maximum principles under partial information
are given. The results are applied to solve a recursive utility optimal problem.

\end{abstract}

\bigskip\textsf{Keywords: Optimal control; Stochastic delay equation;
Mean-field; Stochastic maximum principle; Hamiltonian; Advanced stochastic
equation; Partial information. \newline\newline2010 Mathematics Subject
Classification: \newline Primary 93EXX; 93E20; 60J75; 34K50 \newline Secondary
60H10; 60H20; 49J55}

\section{Introduction}

Stochastic differential equations involving a large number of interacting
particles can be approximated by mean-field stochastic differential equations
(MFSDEs). Solutions of MFSDEs typically occurs as a limit in law of an
increasing number of identically distributed interacting processes, where the
coefficient depends on an average of the corresponding processes. See e.g.
\cite{carmona}. Even more general MFSDEs with delay can be used to model brain
activity in the sense of interactions between cortical columns (i.e. large
populations of neurons). As an example in \cite{Touboul}, they consider a
model of the form%
\[
dX(t,r)=f(t,x)dt+\int_{\Gamma}\mathbb{E}[b(r,r^{\prime},x,X(t-\tau
(r,r^{\prime}))(r^{\prime}))]_{x=X(t,r)}\lambda(dr^{\prime})dt+\sigma
(r)dB(t,r)
\]

Such equations are also used in systemic risk theory and other areas as it is
mentioned in \cite{HOS}. In \cite{MS}, they consider a problem of stochastic
optimal control of forward mean-field delayed equations, they derive necessary
and sufficient maximum principles accordingly and by using continuous
dependence theorems, they prove existence and uniqueness of MF-FSDDEs and
MF-ABSDEs. We emphasize that our paper has similarities with \cite{SH} but in
our case we include delay and jumps and also our type of mean-field equation
is different from theirs.

\section{Finite horizon stochastic mean-field optimal control problem}

Consider a complete filtered probability space $\left(  \Omega,\mathfrak{F}%
,\{\mathfrak{F}_{t}\}_{t\geq0},P\right)  $ on which we define a standard
Brownian motion $B(\cdot)$ and an independent compensated Poisson random
measure $\tilde{N}$, such that $%
\begin{array}
[c]{c}%
\tilde{N}(dt,de):=N(dt,de)-\nu(de)dt,
\end{array}
$where $N(dt,de)$ is the jump measure, $\nu$ is the Lévy measure and
$\nu(de)dt$ is the compensator of $N$. The information available to the
controller may be less than the overall information.

Let $\delta>0$. We want to control a process given by a following pair of
FBSDEs with delay%
\begin{equation}%
\begin{array}
[c]{c}%
dX(t)=b(t,\mathbf{X}(t),\pi(t),\omega)dt+\sigma(t,\mathbf{X}(t),\pi
(t),\omega)dB(t)\\
+\int_{%
\mathbb{R}
_{0}}\gamma(t,\mathbf{X}(t^{\mathbf{-}}),\pi(t^{\mathbf{-}}),e,\omega
)\tilde{N}(dt,de),t\in\lbrack0,T],
\end{array}
\label{eq1}%
\end{equation}

\begin{equation}%
\begin{array}
[c]{c}%
dY(t)=-g(t,\mathbf{X}(t),Y(t),Z(t),\pi(t),\omega)dt+Z(t)dB(t)\\
+\int_{%
\mathbb{R}
_{0}}K(t,e,\omega)\tilde{N}(de,dt),t\in\lbrack0,T],
\end{array}
\label{EQ1}%
\end{equation}

with initial condition $X(t)=X_{0}(t),t\in\lbrack-\delta,0]$ and terminal
condition $Y(T)=aX(T)$ which $a$ is a given constant in $\mathbb{R}_{0},$
where
\[
\mathbf{X}(t):=(1,...,N):=\Big(\int_{-\delta}^{0}X(t+s)\mu_{1}(ds),\ldots
,\int_{-\delta}^{0}X(t+s)\mu_{3}(ds)\Big)
\]
for bounded Borel measures $\mu_{i}(ds),i=1,2,3$ that are either Dirac
measures or absolutely constant.

\begin{example}
Suppose $\mu:=\mu_{1}$ {\normalsize and }${\normalsize N=1}$

\begin{enumerate}
\item If $\mu$ is the Dirac measure concentrated at $0$, then $\mathbf{X}%
(t):=X(t)$, and the state equation is a SDE.

\item If $\mu$ is the Dirac measure concentrated at $-\delta$, then
$\mathbf{X}(t):=X(t-\delta)$, and the state equation is a SDE with discrete delay.

\item If $\mu(ds)=e^{\lambda s}ds$, then $\mathbf{X}(t):=\int_{-\delta}%
^{0}e^{\lambda s}X(s)ds$, and the state equation is a SDE with distributed delay.
\end{enumerate}
\end{example}

Here
\begin{align*}
b  &  =\,\,b(\omega,t,\mathbf{x},\pi)\,\,\,\,\,:\Omega\times\lbrack
0,T]\times\mathbb{R}^{3}\times\times U\longrightarrow\mathbb{R},\\
\sigma &  =\sigma(\omega,t,\mathbf{x},\pi)\,\,\,:\Omega\times\lbrack
0,T]\times\mathbb{R}^{3}\times U\longrightarrow\mathbb{R},\\
\gamma &  =\gamma(\omega,t,\mathbf{x},\pi,e)\,:\Omega\times\lbrack
0,T]\times\mathbb{R}^{3}\times U\times\mathbb{R}_{0}\longrightarrow
\mathbb{R},\\
g  &  =g(\omega,t,\mathbf{x},y,z,k(\cdot),\pi)\,:\Omega\times\lbrack
0,T]\times\mathbb{R}^{3}\times\mathbb{R}\times L^{2}(\mathbb{R}_{0})\times
U\longrightarrow\mathbb{R},\\
K  &  =K(\omega,t,e):\Omega\times\lbrack0,T]\times\mathbb{R}_{0}%
\rightarrow\mathbb{R}.
\end{align*}

Let $b,\sigma,g$ and $\gamma$ are given $\mathfrak{F}_{t}$-measurable for all
$\mathbf{x,}y,z\in\mathbb{%
\mathbb{R}
},$ $u\in U$ where $U$ is a convex subset of $%
\mathbb{R}
$\ and $e\in\mathbb{R}_{0}.$ The set $U$ consists of the admissible control
values. The information available to the controller is given by a
sub-filtration $\mathbb{G}=\{\mathfrak{G}_{t}\}_{t\geq0}$ such that
$\mathfrak{G}_{t}\subseteq\mathfrak{F}_{t}.$ The set of admissible controls,
i.e. the strategies available to the controller is given by a subset
$\mathcal{A}_{\mathbb{G}}$ of the càdlàg, $U$-valued and $\mathfrak{G}_{t}%
$-adapted processes in $L^{2}(\Omega\times\lbrack0,T]).$

\textbf{Assumption (I)}

\begin{description}
\item[i)] The functions $b,\sigma,\gamma$ and $g$ are assumed to be $C^{1}$
(Fréchet) for each fixed $t,$ $\omega$ and $e$.

\item[ii)] \textit{Lipschitz condition:} The functions $b,\sigma$ and $g$ are
Lipschitz continuous in the variables $\mathbf{x},y,z,$ with the Lipschitz
constant independent of the variables $t,u,w$. Also, there exits a function
$L\in L^{2}(\nu)$ independent of $t,u,w$, such that%
\begin{equation}%
\begin{array}
[c]{c}%
\left\vert \gamma(t,\mathbf{x},u,e,\omega)-\gamma(t,\mathbf{x}^{\prime
},u,e,\omega)\right\vert \\
\leq L(e)\left\vert \mathbf{x}-\mathbf{x}^{\prime}\right\vert \mathbf{.}%
\end{array}
\end{equation}

\item[iii)] \textit{Linear growth:} The functions $b,\sigma,g$ and $\gamma$
satisfy the linear growth condition in the variables with the linear growth
constant independent of the variables $t,u,w$. Also there exists a
non-negative function $L^{\prime}\in L^{2}(\nu)$ independent of $t,u,w$, such
that%
\begin{equation}
\left\vert \gamma(t,\mathbf{x},u,e,\omega)\right\vert \leq L^{\prime
}(e)\left(  1+\left\vert \mathbf{x}\right\vert \right)  .
\end{equation}

\end{description}

The optimal control associated to this problem is to optimize the objective
function of the form%

\begin{align}
J(u)  &  =\mathbb{E}\left[  \int_{0}^{T}f(t,\mathbf{X}(t),\mathbb{E[}%
\Phi(X(t))],Y(t),Z(t),K(t,\cdot),\pi(t),\omega)dt\right. \label{EQ2}\\
&  \left.  +h_{1}(Y(0))+h_{2}(X(T),\mathbb{E[}\psi(X(T))])\right]  ,\nonumber
\end{align}

over the admissible controls, for functions%

\[%
\begin{array}
[c]{l}%
f:[0,T]\times\mathbb{R}^{6}\times\mathcal{\phi\times U}\times\Omega
\rightarrow\mathbb{R},\\
\Phi:[0,T]\times\mathbb{R}\times\Omega\rightarrow\mathbb{R},\\
h_{1}:\mathbb{R}\times\Omega\rightarrow\mathbb{R},\\
h_{2}:\mathbb{R}\times\mathbb{R}\times\Omega\rightarrow\mathbb{R},\\
\psi:\mathbb{R}\times\Omega\rightarrow\mathbb{R}.
\end{array}
\]

That is, to find an optimal control $u^{\ast}$ $\in\mathcal{A}_{\mathbb{G}}$
such that%
\begin{equation}
J(u^{\ast})=\underset{u\in\mathcal{A}_{\mathbb{G}}}{\sup}J(u). \label{EQ3}%
\end{equation}

For now, the functions $f,\Phi,\psi$ $,h_{i},i=1,2$ are assumed to satisfy the
following assumptions.

\textbf{Assumption (II)}

\begin{description}
\item[i)] The functions $f(t,\cdot,\omega),\Phi(t,\cdot,\omega),\psi
(\cdot,\omega)$ $,h_{i}(t,\cdot,\omega),i=1,2$ are $C^{1}$ for each $t$ and
$\omega$.

\item[ii)] Integrability condition%
\begin{equation}%
\begin{array}
[c]{l}%
\mathbb{E}\left[  {\int\limits_{0}^{T}}\left\{  \left\vert f\left(
t,\mathbf{X}(t),\mathbb{E[}\Phi(X(t))],Y(t),Z(t),K(t,\cdot),\pi(t)\right)
\right\vert \right.  \right. \\
+\left.  \left.  \left\vert \dfrac{\partial f}{\partial x_{i}}\left(
t,\mathbf{X}(t),\mathbb{E[}\Phi(X(t))],Y(t),Z(t),K(t,\cdot),\pi(t)\right)
\right\vert ^{2}\right\}  dt\right]  <\infty.
\end{array}
\end{equation}

\end{description}

\subsection{The Hamiltonian and adjoint equations}

Let $\mathcal{\phi}$ denote the set of (equivalence classes) measurable
functions $r:%
\mathbb{R}
\rightarrow%
\mathbb{R}
_{0}$ such that%
\begin{equation}%
\begin{array}
[c]{c}%
{\displaystyle\int_{\mathbb{R}_{0}}}
\left\{  \underset{(\mathbf{x},u)\in L}{\sup}\left\vert \gamma(t,\mathbf{x}%
,u,e,\omega)r(e)\right\vert \right. \\
\text{ \ \ \ \ \ \ \ \ \ \ \ \ \ \ \ \ \ \ \ \ \ \ \ \ \ }\left.
+\underset{(\mathbf{x},u)\in L}{\sup}\left\vert \nabla\gamma(t,\mathbf{x}%
,u,e,\omega)r(e)\right\vert \right\}  \nu(de)<\infty
\end{array}
\end{equation}

for each $t\in\lbrack0,T]$ and every bounded $L\subset\mathbb{R}\times U,$
$P-$a.s. This integrability condition ensures that whenever $r\in
\mathcal{\phi},$%
\begin{equation}
\nabla\int_{%
\mathbb{R}
_{0}}\gamma(t,\mathbf{x},u,e,\omega)r(e)\nu(de)=\int_{%
\mathbb{R}
_{0}}\nabla\gamma(t,\mathbf{x},u,e,\omega)r(e)\nu(de),
\end{equation}

and similarly for $K(t,\cdot)$.

\begin{example}
We notice that if the linear growth condition%
\[
\left\vert \gamma(t,\mathbf{x},u,e,\omega)\right\vert +\left\vert \nabla
\gamma(t,\mathbf{x},u,e,\omega)\right\vert \leq L(e)\left\{  1+\left\vert
\mathbf{x}\right\vert +\left\vert u\right\vert \right\}
\]

holds for some $L\in L^{2}(\nu)$ independent of $t,w,$ then $L^{2}(\nu
)\subset\mathcal{\phi},$ and this will be the case in section necessary
maximum principle.
\end{example}

Now we define the Hamiltonian associated to this problem, for $\Omega
\times\lbrack0,T]\times\mathbb{R}^{N}\times\mathbb{R}\times\mathbb{R}%
\times\mathbb{R}\times L^{2}(\nu)\times U\times\mathbb{R}\times\mathbb{R}%
\times L^{2}(\nu)\times\mathbb{R}\rightarrow\mathbb{R};$%

\begin{equation}%
\begin{array}
[c]{l}%
H(t,\mathbf{x,}m\mathbf{,}y,z,k(\cdot),u,p,q,r(\cdot),\lambda)\\
=f(t,\mathbf{x},m,y,z,k,u)+b(t,\mathbf{x},u)p+\sigma(t,\mathbf{x},u)q\\
+g(t,\mathbf{x,}y,z,u)\lambda+\gamma(t,\mathbf{x,}u,e)r(e)\nu(de)
\end{array}
\label{EQ4}%
\end{equation}

where $m=\mathbb{E[}\psi(x)].$

The adjoint equations for all $t\in\lbrack0,T]$ are defined as follows%

\begin{align}
dp(t) &  =\mathbb{E[}\Upsilon\mathbb{(}t)\mid\mathfrak{F}_{t}\mathbb{]}%
dt+q(t)dB(t)+{\int\limits_{\mathbb{R}_{0}}}r(t,e)\tilde{N}(dt,de),\label{EQ5}%
\\
d\lambda(t) &  =\frac{\partial H}{\partial y}\left(  t\right)  dt+\frac
{\partial H}{\partial z}\left(  t\right)  dB(t)+{\int\limits_{\mathbb{R}_{0}%
}\nabla_{k}H}(t,e)\tilde{N}(dt,de),\label{EQ6}%
\end{align}

with, terminal condition%
\begin{align*}
p(T)  &  =a\lambda(T)+\frac{\partial h_{2}}{\partial x}(X(T),\mathbb{E[}%
\psi(X(T))])\\
&  +\frac{\partial h_{2}}{\partial n}(X(T),\mathbb{E[}\psi(X(T))])\psi
^{\prime}(X(T)),a\in\mathbb{R}_{0},
\end{align*}

initial condition $\lambda(0)=h_{1}^{^{\prime}}(Y(0))$ for all $t\in
\lbrack0,T],$ and%

\begin{equation}
\Upsilon(t)=-\sum_{i=0}^{2}\int_{-\delta}^{0}\left\{  \frac{\partial
H}{\partial x_{i}}\Big(t-s,\pi\Big)\right\}  \mu_{i}(ds)-\mathbb{E}%
\Big[\frac{\partial H}{\partial m}\Big(t-s,\pi\Big)\Big]\Phi^{\prime
}(X(t)).\label{EQ7}%
\end{equation}

we denote by $\dfrac{\partial H}{\partial x_{i}},\dfrac{\partial H}{\partial
y},\dfrac{\partial H}{\partial m},\dfrac{\partial h_{2}}{\partial x}%
,\dfrac{\partial h_{2}}{\partial n},$ the partial derivatives of $H$ and $h$
w.r.t. $(x,y,m)$ and $x_{i},n$ resp. and ${\nabla_{k}H}$\ is the Fréchet
derivative of $H$ w.r.t. $k$. Throughout this work, it would be useful to
introduce the simplified notation
\[
h_{2}(T)=h_{2}(X(T),\mathbb{E[}\psi(X(T))]).
\]

\begin{example}
{\normalsize Suppose $\mu:=\mu_{1}$ }

\begin{enumerate}
\item {\normalsize If $\mu$ is the Dirac measure concentrated at $0$, then
\[
\Upsilon(t):=-\frac{\partial H}{\partial x}(t,\pi)-\mathbb{E}\Big[\frac
{\partial H}{\partial m}(t,\pi)\Big].
\]
}

\item {\normalsize If $\mu$ is the Dirac measure concentrated at $\delta$,
then
\[
\Upsilon(t):=-\frac{\partial H}{\partial x}(t+\delta,\pi)-\mathbb{E}%
\Big[\frac{\partial H}{\partial m}(t+\delta,\pi)\Big].
\]
}

\item {\normalsize If $\mu(ds)=e^{\lambda s}ds$, then
\begin{align}
\Upsilon(t):  &  =-\int_{-\delta}^{0}\frac{\partial H}{\partial x}%
(t-s,\pi)e^{\lambda s}(ds)-\mathbb{E}\Big[\frac{\partial H}{\partial m}%
(t,\pi)\Big]\\
\color{red}  &  =-\int_{t-\delta}^{t}\frac{\partial H}{\partial x}%
(-s,\pi)e^{\lambda(s-t)}(ds)-\mathbb{E}\Big[\frac{\partial H}{\partial
m}(t,\pi)\Big]\color{black}.
\end{align}
}
\end{enumerate}
\end{example}

\begin{remark}
The existence and uniqueness of mean-field FBSDEs with delay is beyond the
scope of this paper, and is a topic for future research. We refer to
\cite{Elin}.
\end{remark}

\subsection{A sufficient maximum principle}

When the Hamiltonian $H$ and the functions $(h_{i})_{i=1,2}$ are concave,
under certain other limitations, it is also possible to derive a sufficient
maximum principle.

\begin{theorem}
Let $\hat{\pi}\in\mathcal{A}_{\mathbb{G}}$ with corresponding state processes
$\hat{X},\hat{Y},\hat{Z},\hat{K}(\cdot)$ and adjoint processes $\hat{p}%
,\hat{q},\hat{r}(\cdot)$ and $\hat{\lambda}$. Suppose the following holds:

\begin{enumerate}
\item[1.] (Concavity) The functions%
\begin{equation}%
\mathbb{R}
\times%
\mathbb{R}
\ni(x,n)\longmapsto h_{i}(x,n),i=1,2
\end{equation}
and%
\begin{equation}
\mathbb{R}^{10}\times L^{2}(\nu)\times U\ni(\mathbf{x},m,y,z,k,u)\longmapsto
H(t,\cdot,\hat{p}(t),\hat{q}(t),\hat{r}(t,\cdot),\hat{\lambda}(t))
\end{equation}

\end{enumerate}

are concave $P-$a.s. for each $t\in\lbrack0,T].$

\begin{enumerate}
\item[2.] (Maximum principle)%
\begin{align}
&  \mathbb{E}\left[  H\left(  t,\mathbf{\hat{X}}(t),\hat{M}(t),\hat{Y}%
(t),\hat{Z}(t),\hat{K}(t,\cdot),\hat{\pi}(t),\hat{p}(t),\hat{q}(t),\hat
{r}(t,\cdot),\hat{\lambda}(t)\right)  \Big\vert\mathfrak{G}_{t}\right]
\nonumber\\
&  =\underset{v\in U}{\sup}\mathbb{E}\left[  H\left(  t,\mathbf{\hat{X}%
}(t),\hat{M}(t),\hat{Y}(t),\hat{Z}(t),\hat{K}(t,\cdot),v,\hat{p}(t),\hat
{q}(t),\hat{r}(t,\cdot),\hat{\lambda}(t)\right)  \Big\vert\mathfrak{G}%
_{t}\right]  , \label{EQ8}%
\end{align}

\end{enumerate}

$P-$a.s. for each $t\in\lbrack0,T].$

Then $\pi$ is an optimal control for the problem $\left(  \ref{EQ3}\right)  .$
\newline
\end{theorem}

\begin{proof}
By considering a suitable increasing family of stopping times converging to $T$ we may assume that all the local martingales appearing in the proof below are martingales. See the proof of Theorem 2.1 in \cite{OS2} for details. \\
Let $\pi$ be an arbitrary admissible control. Consider the difference
\begin{align}
&
\begin{array}
[c]{l}%
J(\hat\pi)-J(\pi)\\
=\mathbb{E}\left[
{\displaystyle\int_{0}^{T}}
\left\{  f(t,\mathbf{\hat{X}}(t),\mathbb{E[}\Phi(\hat{X}(t))],\hat{Y}%
(t),\hat{Z}(t),\hat{K}(t,\cdot),\hat\pi(t))\right.  \right. \\
\left.  -f(t,\mathbf{X}(t),\mathbb{E[}\Phi(X(t))],Y(t),Z(t),K(t,\cdot
),\pi(t))\right\}  dt\\
\left.  +h_{1}(\hat{Y}(0))-h_{1}(Y(0))+\hat{h}_{2}(T)-h_{2}(T)\right]
\end{array}
\nonumber\\
&
\begin{array}
[c]{c}%
=\mathbb{E}\left[
{\displaystyle\int_{0}^{T}}
\triangle\hat{f}(t)dt+\triangle\hat{h}_{1}(\hat{Y}(0))+\triangle\hat{h}%
_{2}(\hat{X}(T))\right]  ,
\end{array}
\label{EQ9}%
\end{align}
use the same simplified notation for $\triangle\hat{H}(t),\triangle
\mathbf{\hat{X}}(t)..$etc.
Since $H$ is concave, we have%
\begin{align}
\triangle \hat{H}(t)& \geq \sum_{i=0}^{2}\int_{-\delta }^{0}\triangle
X(t)\left\{ \frac{\partial H}{\partial x_{i}}\Big(t-s,\pi \Big)\right\} \mu
_{i}(ds)  \notag \\
& +\frac{\partial \hat{H}}{\partial m}(t)\mathbb{E[}\triangle \hat{\Phi}(%
\hat{X}(t))]+\frac{\partial \hat{H}}{\partial y}(t)\triangle \hat{Y}(t)
\notag \\
& +\frac{\partial \hat{H}}{\partial z}(t)\triangle \hat{Z}(t)+{\int\limits_{%
\mathbb{R}_{0}}\nabla _{k}}\hat{H}(t,e)\triangle \hat{K}(t,e)\nu (de)+\frac{%
\partial \hat{H}}{\partial \pi}(t)\triangle \hat{\pi}(t)  \notag \\
& \geq \sum_{i=0}^{2}\int_{-\delta }^{0}\triangle X(t)\left\{ \frac{\partial
H}{\partial x_{i}}\Big(t-s,\pi \Big)\right\} \mu _{i}(ds)  \notag \\
& +\frac{\partial \hat{H}}{\partial m}(t)\mathbb{E}\left[ \hat{\Phi}^{\prime
}(\hat{X}(t))\triangle \hat{X}(t)\right] +\frac{\partial \hat{H}}{\partial y}%
(t)\triangle \hat{Y}(t) \\
& +\frac{\partial \hat{H}}{\partial z}(t)\triangle \hat{Z}(t)+{\int\limits_{%
\mathbb{R}_{0}}\nabla _{k}}\hat{H}(t,e)\triangle \hat{K}(t,e)\nu (de)+\frac{%
\partial \hat{H}}{\partial \pi}(t)\triangle \hat{\pi}(t)  \label{EQ10}
\end{align}
By the concavity of $h_{i}(.)_{i=1,2}$, we find
\begin{equation}
\triangle\hat{h}_{1}(0)\geq\hat{h}_{1}^{\prime}(\hat{Y}(0))\triangle\hat
{Y}(0)=\hat{\lambda}(0)\triangle\hat{Y}(0) \label{EQ11}%
\end{equation}
and%
\begin{equation}
\triangle\hat{h}_{2}(T)\geq\triangle\hat{X}(T)\left(  \frac{\partial\hat
{h}_{2}}{\partial x}(T)+\frac{\partial\hat{h}_{2}}{\partial n}(T)\hat{\psi
}^{\prime}(\hat{X}(T))\right)  \label{EQ12}%
\end{equation}
Apply It\^{o}'s formula to $\hat{\lambda}(0)\triangle\hat{Y}(0)$, we get%
\begin{equation}%
\begin{array}
[c]{l}%
\mathbb{E[}\hat{\lambda}(0)\triangle\hat{Y}(0)]=\mathbb{E}\left[  \hat
{\lambda}(T)\triangle\hat{Y}(T)\right. \\
-%
{\displaystyle\int_{0}^{T}}
\left\{  -\hat{\lambda}(t)\triangle\hat{g}(t)+\triangle\hat{Y}(t)\dfrac
{\partial\hat{H}}{\partial y}(t)\right. \\
\left.  \left.  +\triangle\hat{Z}(t)\dfrac{\partial\hat{H}}{\partial z}(t)+%
{\displaystyle\int\limits_{\mathbb{R}_{0}}}
\nabla_{k}\hat{H}(t,e)\triangle\hat{K}(t,e)\nu(de)\right\}  dt\right] \\
=\mathbb{E}\left[  \triangle\hat{X}(T)\left(  \hat{p}(T)-\dfrac{\partial
\hat{h}_{2}}{\partial x}(T)-\dfrac{\partial\hat{h}_{2}}{\partial n}%
(T)\hat{\psi}^{\prime}(\hat{X}(T))\right)  \right. \\
+{\displaystyle\int_{0}^{T}}\left\{  \hat{\lambda}(t)\triangle\hat{g}(t)-\dfrac{\partial
\hat{H}}{\partial y}(t)\triangle\hat{Y}(t)\right. \\
\left.  \left.  -\dfrac{\partial\hat{H}}{\partial z}(t)\triangle\hat{Z}(t)-%
{\displaystyle\int\limits_{\mathbb{R}_{0}}}
\nabla_{k}\hat{H}(t,e)\triangle\hat{K}(t,e)\nu(de)\right\}  dt\right] \\
=\mathbb{E}\left[
{\displaystyle\int_{0}^{T}}
\left\{  \hat{p}(t)\triangle\hat{b}(t)+\triangle\hat{X}(t)\mathbb{E[}\hat{\Upsilon}(t)\mid\mathfrak F_{t}]\mathbb{]+}\triangle\hat{\sigma}(t)\hat{q}(t)\right.
\right. \\
+\int\limits_{\mathbb{R}_{0}}\triangle\hat{\gamma}(t,e)\hat{r}(t,e)\nu
(de)+\hat{\lambda}(t)\triangle\hat{g}(t)-\triangle\hat{X}(T)\left(
\dfrac{\partial\hat{h}_{2}}{\partial x}(T)-\dfrac{\partial\hat{h}_{2}%
}{\partial n}(T)\hat{\psi}^{\prime}(\hat{X}(T))\right) \\
\left.  \left.  -\dfrac{\partial\hat{H}}{\partial y}(t)\triangle\hat
{Y}(t)-\dfrac{\partial\hat{H}}{\partial z}(t)\triangle\hat{Z}(t)-%
{\displaystyle\int\limits_{\mathbb{R}_{0}}}
\nabla_{k}\hat{H}(t,e)\triangle\hat{K}(t,e)\nu(de)\right\}  dt\right] \\
=\mathbb{E}\left[
{\displaystyle\int_{0}^{T}}
\left\{  \triangle\hat{H}(t)-\triangle\hat{f}(t)+\triangle\hat{X}%
(t){\hat{\Upsilon}(}t)-\triangle\hat{X}(T)\left(  \dfrac{\partial\hat
{h}_{2}}{\partial x}(T)-\dfrac{\partial\hat{h}_{2}}{\partial n}(T))\hat{\psi
}^{\prime}(\hat{X}(T))\right)  \right.  \right. \\
\left.  \left.  -\dfrac{\partial\hat{H}}{\partial y}(t)\triangle\hat
{Y}(t)-\dfrac{\partial\hat{H}}{\partial z}(t)\triangle\hat{Z}(t)-%
{\displaystyle\int\limits_{\mathbb{R}_{0}}}
\nabla_{k}\hat{H}(t,e)\triangle\hat{K}(t,e)\nu(de)\right\}  dt\right]  .
\end{array}
\label{EQ25}
\end{equation}
By the definition of $\Upsilon\left(  \ref{EQ7}\right)  $\ and Fubini's theorem, we can show that
\begin{eqnarray}
&&\int_{0}^{T}\int_{-\delta }^{0}\triangle X(t)\left\{ \frac{\partial H}{%
\partial x_{i}}\Big(t-s,\pi \Big)\right\} \mu _{i}(ds)dt  \notag \\
&=&\int_{0}^{T}\int_{-\delta }^{0}\left\{ \frac{\partial H}{\partial x_{i}}%
\Big(t,\pi \Big)\right\} \triangle X(t+\delta )\mu _{i}(ds)dt  \label{nacira}
\end{eqnarray}
Let's perform the change of variable $r=t-s$ in the dt-integral to observe that
\small\begin{equation}
\begin{split}
\mathbb E\Big[&\int_0^T  \int_{-\delta}^0 X(t)\Big( \frac{\partial H}{\partial x_i}(t-s,\pi)\Big)\mu_i(ds) dt\Big]\\
&=\mathbb E\Big[\int_{-\delta}^0\int_0^T  X(t)\Big( \frac{\partial H}{\partial x_i}(t-s,\pi)\Big)dt\,\mu_i(ds) \Big]\\
&=\mathbb E\Big[\int_{-\delta}^0\int_s^T  X(t)\Big( \frac{\partial H}{\partial x_i}(t-s,\pi)\Big)dt\,\mu_i(ds) \Big]\\
&=\mathbb E\Big[\int_{-\delta}^0\int_0^T X(r+s)\Big( \frac{\partial H}{\partial x_i}(r,\pi) \Big)dt\,\mu_i(ds) \Big]\\
\color{red}
&=\mathbb E\Big[\int_0^T \int_{-\delta}^0 X(t+s)\Big( \frac{\partial H}{\partial x_i}(t,\pi) \Big)\mu_i(ds)\,dt \Big]\\
&=\mathbb E\Big[\int_{0}^T\int_{-\delta}^0 X(t+s)\Big( \frac{\partial H}{\partial x_i}(t,\pi) \Big)\,\mu_i(ds)dt \Big]
\end{split}
\end{equation}\normalsize
Putting $\left(  \ref{nacira}\right)  $ in $\left(  \ref{EQ25}\right)  $, and combining $\left(  \ref{EQ10}\right)
$ with $\left(  \ref{EQ9}\right)  $, we obtain%
\[
J(\hat\pi)-J(\pi)\geq\mathbb{E}\left[  \int_{0}^{T}\frac{\partial\hat{H}%
}{\partial \pi}(t)\triangle\hat\pi(t)dt\right]  \geq0,
\]
by the maximum condition of $H$ $\left(  \ref{EQ8}\right)  $.
\end{proof}

\section{Infinite horizon optimal control problem}

In this section, we extend the results obtained in the previous section to
infinite horizon. So it can be seen as a generalization to mean-field problems
of Theorems {\normalsize $3.1$ and $4.1$ } in \cite{DA} and \cite{Peng} resp.
By following the same steps in the previous section but now with infinite time
horizon, we consider that the state equations have the forms%

\begin{equation}
\left.
\begin{array}
[c]{l}%
dX(t)=b(t,\mathbf{X}(t),\mathbb{E[}\mathbf{X}(t)],\pi(t),\omega)dt+\sigma
(t,\mathbf{X}(t),\mathbb{E[}\mathbf{X}(t)],\pi(t),\omega)dB(t)\\
+\int_{%
\mathbb{R}
_{0}}\gamma(t,\mathbf{X}(t^{\mathbf{-}}),\mathbb{E[}\mathbf{X}(t^{\mathbf{-}%
})],\pi(t^{\mathbf{-}}),e,\omega)\tilde{N}(dt,de),t\in\lbrack0,\infty),\\
X(t)=x_{0}(t),t\in\left[  -\delta,0\right]  ,
\end{array}
\right.  , \label{EQ27}%
\end{equation}
and%
\begin{equation}%
\begin{array}
[c]{l}%
dY(t)=-g(t,\mathbf{X}(t),\mathbb{E[}\mathbf{X}(t)],Y(t),\mathbb{E[}%
Y(t)],Z(t),K(t,\cdot),\pi(t))dt+Z(t)dB(t)\\
+\int_{%
\mathbb{R}
_{0}}K(t,e,\omega)\tilde{N}(dt,de),t\in\lbrack0,\infty)
\end{array}
\label{EQ28}%
\end{equation}

which can be interpreted as in \cite{Pardoux} for all finite $T,$%
\begin{equation}%
\begin{array}
[c]{l}%
Y(t)=Y(T)+\int_{t}^{T}g(s,\mathbf{X}(s),\mathbb{E[}\mathbf{X}%
(s)],Y(s),\mathbb{E[}Y(s)],Z(s),K(s,\cdot),\pi(s))ds\\
+\int_{%
\mathbb{R}
_{0}}K(t,e,\omega)\tilde{N}(dt,de),0\leq t\leq T,
\end{array}
\end{equation}

where
\[
\mathbf{X}(t):=(X_{1}(t),X_{2},(t),\ldots X_{N}(t)):=\Big(\int_{-\delta}%
^{0}X(t+s)\mu_{1}(ds),\ldots,\int_{-\delta}^{0}X(t+s)\mu_{N}(ds)\Big)
\]
for bounded Borel measures $\mu_{1},\ldots\mu_{N}(ds)$. We remark if $X$ is a
càdlàg process, then $\mathbf{X}$ is also càdlàg. We always assume that
coefficient functional $\gamma$ is evaluated for the predictable ( i.e. left
continuous) versions of the càdlàg processes $\mathbf{X},Y$ and $\pi$, and we
will omit the minus from the notation.
\begin{align*}
b=\,\,b(\omega,t,\mathbf{x},\mathbf{m},u)\,\,\,\,\,:  &  \Omega\times
\lbrack0,\infty)\times\mathbb{R}^{N}\times{\mathbb{R}}^{N}\times
U\longrightarrow\mathbb{R},\\
\sigma=\sigma(\omega,t,\mathbf{x},\mathbf{m},u)\,\,\,\,\,:  &  \Omega
\times\lbrack0,\infty)\times\mathbb{R}^{N}\times\mathbb{R}^{N}\times
U\longrightarrow\mathbb{R},\\
\gamma=\gamma(\omega,t,\mathbf{x},\mathbf{m},u,e):  &  \Omega\times
\lbrack0,\infty)\times\mathbb{R}^{N}\times{\mathbb{R}}^{N}\times
U\times\mathbb{R}_{0}\longrightarrow\mathbb{R},\\
g:=g(\omega,t,\mathbf{x},\mathbf{m},y,n,z,k(\cdot),u):  &  \Omega\times
\lbrack0,\infty)\times\mathbb{R}^{N}\times{\mathbb{R}}^{N}\times
\mathbb{R}\times\mathbb{R}\times\mathbb{R}\times L^{2}(\nu)\times
U\longrightarrow\mathbb{R}.
\end{align*}
We assume that the coefficient functional satisfy the following assumptions:

\textbf{Assumptions (III)}

\begin{enumerate}

\item The functions $b,\sigma, \gamma,g$ is $C^{1}$ (Fr\'echet) with respect
to all variables except $t$ and $\omega$.

\item The functions $b,\sigma,\gamma,g$ are jointly measurable.
\end{enumerate}

Let $U$ be a subset of $\mathcal{\phi}$. The set $U$ will be the admissible
control values. The information available to the controller is given by a
sub-filtration $\mathbb{G}=\{\mathfrak{G}_{t}\}_{t\in\lbrack0,T]}$ with
$\mathfrak{F}_{0}\subset\mathfrak{G}_{t}\subset\mathfrak{F}_{t}$.

The set of admissible controls, that is, the set of controls that are
available to the controller, is denoted by $\mathcal{A}_{\mathbb{G}}$. It will
be a given subset of the càdlàg, $U$-valued and $\mathfrak{G}_{t}$-adapted
processes in $L^{2}(\Omega\times\lbrack0,\infty))$, such that there exists
unique càdlàg adapted processes $X=X^{\pi},Y=Y^{\pi}$, progressively
measurable $Z=Z^{\pi}$, and predictable $K=K^{\pi}$ satisfying (\ref{EQ27})
and (\ref{EQ28}), and if it also satisfyes {\small
\begin{equation}
\mathbb{E}\Big[\int_{0}^{\infty}\left\vert X(s)\right\vert ^{2}%
ds\Big]+\mathbb{E}\left[  \underset{t\geq0}{\sup\text{ }}e^{\kappa t}\left(
Y(t)\right)  ^{2}+%
{\textstyle\int\limits_{0}^{\infty}}
e^{\kappa t}(\left(  Z(t)\right)  ^{2}+%
{\textstyle\int\limits_{\mathbb{R}_{0}}}
\left(  K\left(  t,e\right)  \right)  ^{2}\nu(de))dt\right]  <\infty
\label{growth}%
\end{equation}
}{\normalsize for some constant $\kappa>0.$ }

\subsection{{\protect\normalsize The Optimization problem}}

{\normalsize We want to maximise the performance functional }{\small
\[
J(\pi)=\mathbb{E}\left[  \int_{0}^{\infty}f(t,\mathbf{X}(t),\mathbb{E[}%
\mathbf{X}(t)],Y(t),\mathbb{E[}Y(t)],Z(t),K(t,\cdot),\pi(t))dt+h(Y(0))\right]
\]
}{\normalsize over the set $\mathcal{A}_{\mathbb{G}}$, for some functions
\[
f=f(\omega,t,\mathbf{x},\mathbf{m},y,n,z,k(\cdot),u):\Omega\times
\lbrack0,\infty)\times\mathbb{R}^{N}\times\mathbb{R}^{N}\times\mathbb{R}%
\times\mathbb{R}\times\mathcal{\phi}\times U\rightarrow\mathbb{R},
\]
and
\[
h=:\mathbb{R}\rightarrow\mathbb{R}.
\]
That is, we want to find $\pi^{\ast}\in\mathcal{A}_{\mathbb{G}}$ such that
\begin{equation}
\underset{\pi\in\mathcal{A}_{\mathbb{G}}}{\sup}J(\pi)=J(\pi^{\ast}).
\label{EQ30}%
\end{equation}
We assume that the functions $f$ and $h$ satisfy the following assumptions: }

{\normalsize \textbf{Assumptions (IV)} }

\begin{enumerate}
{\normalsize
}

\item {\normalsize The functions $f,h$ is $C^{1}$ (Fr\'echet) with respect to
all variables except $t$ and $\omega$. }

\item {\normalsize The function $f$ is predictable, and $h$ is $\mathfrak{F}%
$-measurable for fixed $\mathbf{x}, \mathbf{m}, y, n, z, k, u$. }
\end{enumerate}

\subsection{{\protect\normalsize The Hamiltonian and the adjoint equation}}

{\normalsize Define the Hamiltonian function
\[
H:\Omega\times\lbrack0,\infty)\times\mathbb{R}^{N}\times\mathbb{R}^{N}%
\times\mathbb{R}\times\mathbb{R}\times\mathbb{R}\times L^{2}(\nu)\times
U\times\mathbb{R}\times\mathbb{R}\times L^{2}(\nu)\times\mathbb{R}%
\rightarrow\mathbb{R}%
\]
by
\begin{equation}%
\begin{array}
[c]{l}%
H(t,\mathbf{x,}\mathbf{m}\mathbf{,}y,n,z,k(\cdot),u,p,q,r(\cdot),\lambda)\\
=b(t,\mathbf{x},\mathbf{m},u)p+\sigma(t,\mathbf{x},\mathbf{m},u)q+%
{\displaystyle\int_{\mathbb{R}_{0}}}
\gamma(t,\mathbf{x,}m\mathbf{,}u,e)r(e)\nu(de)\\
+g(t,\mathbf{x,}\mathbf{m}\mathbf{,}y,n,z,k(\cdot),u)\lambda+f(t,\mathbf{x}%
,\mathbf{m},y,n,z,k(\cdot),u).
\end{array}
\label{haminf}%
\end{equation}
Now, to each admissible control $\pi$, we can define the adjoint processes
$p,q,r$ and $\lambda$ by the following system of forward-backward
equations:\newline}

{\normalsize \textbf{Backward equation} }%
\begin{equation}
\
dp(t)=-\mathbb{E}[\Upsilon(t)|\mathcal{F}_{t}]dt+q(t)dB(t)+\int_{\mathbb{R}%
_{0}}r(t,e)\tilde{N}(dt,de),
\
\label{eq:pqr}
\end{equation}
where{\small
\begin{equation}%
\begin{split}
\Upsilon(t)  &  =\sum_{i=0}^{N-1}\int_{-\delta}^{0}\frac{\partial H}{\partial
x_{i}}\Big(t-s,\mathbf{X}(t-s),\mathbb{E}[\mathbf{X}(t-s)],Y(t-s),\mathbb{E}%
[Y(t-s)],\\
&  \quad\quad Z(t-s),K(t-s),\pi(t-s),p(t-s),q(t-s),r(t-s)\Big)\mu_{i}(ds)\\
&  +\sum_{i=0}^{N-1}\int_{-\delta}^{0}\mathbb{E}\Big[\frac{\partial
H}{\partial m_{i}}\Big(t-s,\mathbf{X}(t-s),\mathbb{E}[\mathbf{X}%
(t-s)],Y(t-s),\mathbb{E}[Y(t-s)],\\
&  \quad\quad Z(t-s),K(t-s),\pi(t-s),p(t-s),q(t-s),r(t-s)\Big)\Big]\mu
_{i}(ds).
\end{split}
\label{eq:Upsilon}%
\end{equation}
}

{\normalsize \textbf{Forward equation} }{\small
\begin{equation}%
\begin{split}
d\lambda(t)  &  =\frac{\partial H}{\partial y}\Big(t,\mathbf{X}(t),\mathbb{E}%
[\mathbf{X}(t)],Y(t),\mathbb{E}[Y(t)],Z(t),K(t),\pi(t),p(t),q(t),r(t),\lambda
(t)\Big)\\
&  +\mathbb{E}\Big[\frac{\partial H}{\partial n}\Big(t,\mathbf{X}%
(t),\mathbb{E}[\mathbf{X}(t)],Y(t),\mathbb{E}[Y(t)],Z(t),K(t),\pi
(t),p(t),q(t),r(t),\lambda(t)\Big)\Big]dt\\
&  +\frac{\partial H}{\partial z}\Big(t,\mathbf{X}(t),\mathbb{E}%
[\mathbf{X}(t)],Y(t),\mathbb{E}[Y(t)],Z(t),K(t),\pi(t),p(t),q(t),r(t),\lambda
(t)\Big)dB(t)\\
&  +\int_{\mathbb{R}_{0}}\nabla_{k}H\Big(t,\mathbf{X}(t),\mathbb{E}%
[\mathbf{X}(t)],Y(t),\mathbb{E}[Y(t)],Z(t),K(t),\pi(t),p(t),q(t),r(t),\lambda
(t)\Big)(e)\tilde{N}(dt,de),\\
\lambda(0)  &  =h^{\prime}(Y(0)).
\end{split}
\label{eq:lambda}%
\end{equation}
}{\normalsize
Here $\nabla_{k}H$ is used to denote the Fréchet derivative of $H$ with
respect to the variable $k(\cdot)$, and hence }{\small
\[
\nabla_{k}H\Big(t,\mathbf{X}(t),\mathbb{E}[\mathbf{X}(t)],Y(t),\mathbb{E}%
[Y(t)],Z(t),K(t),\pi(t),p(t),q(t),r(t),\lambda(t)\Big)\in L^{2}(\nu)^{\ast
}=L^{2}(\nu),
\]
}{\normalsize for fixed $\omega,\pi$ and corresponding $\mathbf{X}%
,Y,Z,K,p,q,r$ and $\lambda$. We notice also that the integrand
\[
\nabla_{k}H\Big(t,\mathbf{X}(t),\mathbb{E}[\mathbf{X}(t)],Y(t),\mathbb{E}%
[Y(t)],Z(t),K(t),\pi(t),p(t),q(t),r(t),\lambda(t)\Big)(e)
\]
is predictable.}

{\normalsize Notice also, $\Upsilon$ may not be adapted to $\mathfrak{F}_{t}$,
as $\Upsilon(t)$ is defined using values of $H$ at time $t-s$, where $s<0$. }

{\normalsize Given an admissible control $\pi$, suppose there exists
progressively measurable processes $p=p^{\pi},q=q^{\pi}$ and $\lambda
=\lambda^{\pi}$ and $r=r^{\pi}$, satisfying (\ref{eq:pqr})-(\ref{eq:lambda})
and such that
\begin{equation}
\mathbb{E}\left[  \underset{\text{ \ \ \ \ \ \ \ \ }}{\underset{t\geq
0}{\text{sup}}}e^{\kappa t}\left(  p(t)\right)  ^{2}+%
{\displaystyle\int\limits_{0}^{\infty}}
\left\{  \left\vert \lambda(t)\right\vert ^{2}+e^{\kappa t}(\left(
q(t)\right)  ^{2}+%
{\textstyle\int\limits_{\mathbb{R}_{0}}}
\left(  r\left(  t,e\right)  \right)  ^{2}\nu(de))\right\}  dt\right]
<\infty\label{decayp}%
\end{equation}
}

{\normalsize for some constant $\kappa>0.$ Then, we say that $p,q,r$ and
$\lambda$ are adjoint equations to the Forward-Backward system (\ref{EQ27}%
)-(\ref{EQ28}).}

\subsection{{\protect\normalsize Short hand notation}}

{\normalsize When Adjoint processes exist, we will frequently use the
following short hand notation: }{\small
\begin{equation}
H(t,\pi):=H\big(t,\mathbf{X}^{\pi}(t),\mathbb{E}[\mathbf{X}^{\pi}(t)],Y^{\pi
}(t),\mathbb{E}[Y^{\pi}(t)],Z^{\pi}(t),K^{\pi}(t),\pi(t),p^{\pi}(t),q^{\pi
}(t),r^{\pi}(t)\big).
\end{equation}
}{\normalsize Similar notation will be used for the coefficient functions
$b,\sigma,\gamma,$ and $g$, and the functions $f,h$ from the performance
functional, and for derivatives of the mentioned functions. We will write
$\nabla H$ for the Fréchet derivative of $H$ with respect to the variables
${\mathbf{x}},{\mathbf{n}},y,n,z,k(\cdot)$. Notice that $\nabla H(t,\pi)$
applied to
\[
(\bar{\mathbf{x}},\bar{\mathbf{n}},\bar{y},\bar{n},\bar{z},\bar{k}(\cdot
),\bar{u})\in\mathbb{R}^{k}\times\mathbb{R}^{k}\times\mathbb{R}\times
\mathbb{R}\times\mathbb{R}\times L^{2}(\nu)\times U
\]
is given by }{\small
\begin{align*}
&  \nabla H(t,\pi)(\bar{\mathbf{x}},\bar{\mathbf{n}},\bar{y},\bar{n},\bar
{z},\bar{k}(\cdot),\bar{u})=\nabla_{\mathbf{x}}H(t,\pi)\bar{\mathbf{x}%
}^{\mathsf{T}}+\nabla_{\mathbf{m}}H(t,\pi)\bar{\mathbf{m}}^{\mathsf{T}}\\
&  \quad+\frac{\partial H}{\partial y}(t,\pi)\bar{y}+\frac{\partial
H}{\partial y}(t,\pi)\bar{n}+\frac{\partial H}{\partial z}(t,\pi)\bar{z}%
\int_{\mathbb{R}_{0}}\nabla_{k}H(t,\pi)(e)\bar{k}(e)\nu(de)+\frac{\partial
H}{\partial u}(t,\pi)\bar{u}.
\end{align*}
}{\normalsize where $\nabla_{\mathbf{x}}$ is the gradient (as a row vector)
with respect to the variable $\mathbf{x}$, etc. }

{\normalsize Using this notation, the state equations and the adjoint
equations can be written more compactly as }{\small
\begin{align}
&
\begin{alignedat}{2} dX(t)&=b(t,\pi)dt+\sigma(t,\pi)dB(t)+\int_{\mathbb{R}_{0}}\gamma(t,\pi,e)\tilde{N}(dt,de),t\in\lbrack0,\infty),\\ X(t)&=x_{0}(t),\quad t\in\left[ -\delta,0\right] , \end{alignedat}\\
&  \,\nonumber\\
&
\begin{alignedat}{2} dY(t)&=-g(t,\pi)dt+Z(t)dB(t)+\int_{\mathbb{R}_{0}}K(t,e,\omega)\tilde{N}(dt,de),\quad t\in\lbrack0,\infty), \end{alignedat}
\end{align}
}{\normalsize and }{\small
\begin{align}
&
\begin{alignedat}{2} dp(t)&=-\mathbb E[\Upsilon(t)|\mathfrak F_t ]\quad+q(t)dB(t)+ \int_{\mathbb R_0}r(e,t)\tilde N(dt, de),\\ &\textnormal{where}\\ \Upsilon(t)&=\sum_{i=0}^{N-1}\int_{-\delta}^0{{\frac{\partial H}{\partial x_i}(t-s,\pi)+ \mathbb E\Big[\frac{\partial H}{\partial m}(t-s,\pi)\Big]}}\mu_i(ds) \end{alignedat}\\
&  \,\nonumber\\
&
\begin{alignedat}{2} d\lambda(t)&= \Big\{\frac{\partial H}{\partial y}(t,\pi)+\mathbb E\Big[\frac{\partial H}{\partial n}(t,\pi)\Big]\Big\} dt\\ &\quad+\dfrac{\partial H}{\partial z}\left( t,\pi\right) dB(t)+\int\limits_{\mathbb{R}_{0}}\nabla_k H(t,\pi)(e)\tilde{N}(dt,de),t\in [0,\infty)\\ \lambda(0)&=h'(Y(0)). \end{alignedat}
\end{align}
}

\begin{example}
{\normalsize Suppose $N=1$, $\mu:=\mu_{1}$ }

\begin{enumerate}
\item {\normalsize If $\mu$ is the Dirac measure concentrated at $0$, then
\[
\Upsilon(t):=\frac{\partial H}{\partial x}(t,\pi)+ \mathbb{E}\Big[\frac
{\partial H}{\partial m}(t,\pi)\Big].
\]
}

\item {\normalsize If $\mu$ is the Dirac measure concentrated at $-\delta$,
then
\[
\Upsilon(t):=\frac{\partial H}{\partial x}(t+\delta,\pi)+ \mathbb{E}%
\Big[\frac{\partial H}{\partial m}(t+\delta,\pi)\Big].
\]
}

\item {\normalsize If $\mu(ds)=e^{\lambda s}ds$, then
\begin{align}
\Upsilon(t):  &  =\int_{-\delta}^{0}\frac{\partial H}{\partial x}(t-s,\pi)+
\mathbb{E}\Big[\frac{\partial H}{\partial m}(t-s,\pi)\Big]e^{\lambda s}(ds)\\
\color{red}  &  =\int_{t-\delta}^{t} \frac{\partial H}{\partial x}(-s,\pi)+
\mathbb{E}\Big[\frac{\partial H}{\partial m}(-s,\pi)\Big]e^{\lambda
(s-t)}(ds)\color{black} .
\end{align}
}
\end{enumerate}
\end{example}

{\normalsize
}

{\normalsize
}

{\normalsize
}

{\normalsize
}

\section{{\protect\normalsize A necessary maximum principle}}

{\normalsize Suppose that a control $\pi\in\mathcal{A}_{\mathbb{G}}$ is
optimal and that $\eta\in\mathcal{A}_{\mathbb{G}}.$ If the function
$s\longmapsto J(\pi+s\eta)$ is well defined and differentiable on a
neighbourhood of $0$, then
\begin{equation}
\frac{d}{ds}J(\pi+s\eta)\mid_{s=0}=0.
\end{equation}
Under a set of suitable assumptions on the functions $f,b,\sigma,g,h,$%
}${\normalsize \gamma}$ {\normalsize and $K$, we will show that for every
admissible $\pi$, and bounded admissible $\eta,$%
\begin{equation}
\frac{d}{ds}J(\pi+s\eta)\mid_{s=0}=\mathbb{E}\left[  \int_{0}^{\infty}%
\frac{\partial}{\partial\pi}H\left(  t,\pi\right)  \eta(t)dt\right]  .
\end{equation}
Then, provided that the set of admissible controls $\mathcal{A}_{\mathbb{G}}$
is sufficiently large,%
\begin{equation}
\frac{d}{ds}J(\pi+s\eta)\mid_{s=0}=0
\end{equation}
is equivalent to%
\begin{equation}
\mathbb{E}\left[  \frac{\partial}{\partial\pi}H\left(  t,\pi\right)
\mathcal{\mid}\mathfrak{G}_{t}\right]  =0\text{ \ }P-\text{a.s. for each }%
t\in\lbrack0,\infty).
\end{equation}
Consequently,%
\begin{equation}
\mathbb{E}\left[  \frac{\partial}{\partial\pi}H\left(  t,\pi\right)
\mathcal{\mid}\mathfrak{G}_{t}\right]  =0\text{ \ }P-\text{a.s. for each }%
t\in\lbrack0,\infty),
\end{equation}
is a necessary condition for optimality of $\pi.$ }

{\normalsize The first step of deriving a necessary maximum principle is to
establish the following equalities.}

{\small
\begin{align*}
&  \frac{d}{ds}J(\pi+s\eta)\mid_{s=0}\\
&  =\mathbb{E}\left[  \int_{0}^{\infty}\nabla f(t,\mathbf{X}^{\pi}%
(t),\pi)\cdot(\varkappa^{\pi}(t),\eta(t))^{T}dt+h^{\prime}(Y^{\pi}%
(0))\cdot\mathcal{Y}\text{ }^{\pi,\eta}(0)\right] \\
&  =\mathbb{E}\Big[\int_{0}^{\infty}\frac{\partial H}{\partial{\normalsize \pi
}}(t,\pi)\eta(t)dt\Big].
\end{align*}
}{\normalsize We will formalize this through Lemma }${\normalsize 4.3}$
{\normalsize and Lemma }${\normalsize 4.6}${\normalsize , but first we need to
impose a set of assumptions: }

{\normalsize \textbf{Assumptions (V)} }

\begin{description}
\item[i)] {\normalsize \textit{Assumptions on the coefficient functions} }

\begin{itemize}
\item {\normalsize The functions $\nabla b,\nabla\sigma$ and $\nabla g$ are
bounded. The upper bound is denoted by $D_{0}$. Also, there exists a
non-negative function $D\in L^{2}$ such that%
\begin{equation}
\left\vert \nabla\gamma(t,\mathbf{x},u,e)\right\vert +\left\vert
K(t,e)\right\vert \leq D(e)
\end{equation}
}

\item {\normalsize The functions $\nabla b,\nabla\sigma$ and $\nabla g$ are
Lipschitz continuous in the variables $\mathbf{x},\mathbf{m},u,$ uniformly in
$t,w,$ with the Lipschitz constant $L_{0}>0$. Also, there exits a function
$L\in L^{2}(\nu)$ independent of $t,w$, such that%
\begin{equation}%
\begin{array}
[c]{c}%
\left\vert \gamma(t,\mathbf{x},u,e,\omega)-\gamma(t,\mathbf{x}^{\prime
},u^{\prime},e)\right\vert \\
\leq L(e)\Big(\left\vert \mathbf{x}-\mathbf{x}^{\prime}\right\vert +\left\vert
u-u^{\prime}\right\vert \Big) .
\end{array}
\end{equation}
}

\item {\normalsize The function $L^{\prime}$ from Assumption (I)}
{\normalsize is also in $L^{2}(\nu)$. }
\end{itemize}

\item[ii)] {\normalsize \textit{Assumptions on the performance functional} }

\begin{itemize}
\item {\normalsize The functions $\nabla f,\nabla h$ and $\nabla g$ are
bounded.The upper bound is still denoted by $D_{0}$. }

\item {\normalsize The functions $\nabla f,\nabla h$ and $\nabla g$ are
Lipschitz continuous in the variables $(\mathbf{x},\mathbf{y},z,k,u),$
uniformly in $t,w.$ The Lipschitz constant is still denoted by $L_{0}$. }
\end{itemize}

\item[iii)] {\normalsize Assumptions on the set of admissible processes }

\begin{itemize}
\item {\normalsize Whenever }${\normalsize u}${\normalsize $\in\mathcal{A}%
_{\mathbb{G}}$ and $\eta\in\mathcal{A}_{\mathbb{G}}$ is bounded, there exists
$\epsilon\mathcal{>}0$ such that%
\begin{equation}
u+s\eta\in\mathcal{A}_{\mathbb{G}}\text{\ \ \ \ for each }s\in(-\epsilon
,\epsilon). \label{EQ14}%
\end{equation}
}

\item {\normalsize For each $t_{0}>0$ and each bounded $\mathfrak{G}_{t_{0}}%
$-measurable random variables $\alpha$, the process $\eta(t)=\alpha1_{\left[
t_{0},t_{0}+h\right)  }(t)$ belongs to $\mathcal{A}_{\mathbb{G}}.\smallskip$ }
\end{itemize}
\end{description}

{\normalsize \color{black} }

\subsection{{\protect\normalsize The derivative processes}}

{\normalsize Suppose that $\pi,\eta\in\mathcal{A}_{\mathbb{G}}$, with $\eta$
bounded. Consider the equations }

{\normalsize
\begin{equation}%
\begin{split}
d\mathcal{X}(t)  &  =\nabla b(t,\pi)\cdot\left(  \mathbb{X}(t),\mathbb{E}%
[\mathbb{X}(t)],\eta(t)\right)  ^{\mathsf{T}}dt\\
&  +\nabla\sigma(t,\pi)\cdot\left(  \mathbb{X}(t),\mathbb{E}[\mathbb{X}%
(t)],\eta(t)\right)  ^{\mathsf{T}}dB(t)\\
&  +\int_{\mathbb{R}_{0}}\nabla\gamma(t,\pi,e)\cdot(\mathbb{X}(t),\mathbb{E}%
[\mathbb{X}(t)],\eta(t))^{\mathsf{T}}\tilde{N}(dt,de),t\in\lbrack0,\infty),\\
\mathcal{X}(t)  &  =0,\text{ \ \ \ for }t\in\left[  -\delta,0\right]  ,
\end{split}
\label{eq:DerivativeX}%
\end{equation}
where
\[
\mathbb{X}(t):=\big(\int_{-\delta}^{0}\mathcal{X}(t+s)\mu_{1}(ds),\ldots
,\int_{-\delta}^{0}\mathcal{X}(t+s)\mu_{N}(ds)\big)
\]
and%
\begin{equation}%
\begin{split}
d\mathcal{Y}(t)  &  =\Big(-\nabla g(t,\pi)\Big)\cdot\Big(\mathbb{X}%
(t),\mathbb{E}[\mathbb{X}(t)],\mathcal{Y}(t),\mathbb{E}[\mathcal{Y}%
(t)],\mathcal{Z}(t),\mathcal{K}(t),\eta(t)\Big)^{\mathsf{T}}dt\\
&  +\mathcal{Z}(t)dB(t)+{\displaystyle\int_{\mathbb{R}_{0}}}\mathcal{K}%
(t,e)\tilde{N}(dt,de),\quad t\in\lbrack0,\infty).
\end{split}
\label{eq:DerivativeYZK}%
\end{equation}
We say that a solutions $\mathcal{X}=\mathcal{X}^{\pi,\eta},\mathcal{Y}%
=\mathcal{Y}^{\pi,\eta},\mathcal{Z}=\mathcal{Z}^{\pi,\eta}$ and $\mathcal{K}%
=\mathcal{K}^{\pi,\eta}$ associated with the controls $\pi,\eta$ exists if
there are processes $\mathcal{X},\mathcal{Y},\mathcal{Z}$ and $\mathcal{K}$
satisfying (\ref{eq:DerivativeX})-(\ref{eq:DerivativeYZK}), and \color{black}
}

{\normalsize
\begin{equation}
\mathbb{E}\left[  \underset{t\geq0}{\sup\text{ }}e^{\kappa t}\left(
\mathcal{Y}(t)\right)  ^{2}+%
{\textstyle\int\limits_{0}^{\infty}}
\left\vert \mathcal{X}(t)\right\vert ^{2}+e^{\kappa t}(\left(  \mathcal{Z}%
(t)\right)  ^{2}+%
{\textstyle\int\limits_{\mathbb{R}_{0}}}
\left(  \mathcal{K}\left(  t,e\right)  \right)  ^{2}\nu(de))dt\right]  <\infty
\end{equation}
for some constant $\kappa>0$. }

\subsubsection{{\protect\normalsize Differentiability of the forward state
process}}

{\normalsize To proofs in this section are similar to e.g. the proofs of
Lemmas $3.1$, $4.1$ and in \cite{DA} and \cite{Peng} resp. However because of
our jump term, we need to use Kunita's inequality instead of
Burkholder-Davis-Gundy's inequality. We also do not require any $L^{4}%
$-boundedness and convergence of any of our processes as is done e.g in
\cite{DA}, to assure the convergence in our Lemma \ref{lemma:Jdifferentable1}.
Requiring $L^{4}$ boundedness on the process would have lead to the necessity
of additional assumptions on the Lipshitz and boundedness constants, as an
example assuming that the function $D$ in assumption (V) is also in $L^{4}%
(\nu)$ is a sufficient condition to ensure that $\mathbb{E}[\sup_{0\leq v\leq
t}|X(v)|^{4}]<\infty$ for each $t\in\lbrack0,\infty)$. For convenience to the
reader, let us recall }

\begin{lemma}
[Kunita's inequality, corollary 2.12 in \cite{Kunita}]{\normalsize Suppose
$\rho\geq2$ and
\begin{align}
X(t)=x+\int_{0}^{t}b(r)dr+\int_{0}^{t}\sigma(r)dB(r)+\int_{0}^{t}
\int_{\mathbb{R}_{0}}\mathcal{X}(r,e)\tilde{N}(dr,de).
\end{align}
Then there exists a positive constant $C_{\rho, T},$ (depending only on $\rho,
T$) such that the following inequality holds%
\begin{align*}
\mathbb{E}\left[  \underset{0\leq s\leq T}{\sup}\left\vert X(s)\right\vert
^{\rho}\right]   &  \leq C_{\rho,T}\left[  \left\vert x\right\vert ^{\rho
}+\left\{
{\displaystyle\int_{0}^{t}}
\mathbb{E}\left[  \left\vert b(r)\right\vert ^{\rho}\right]  +\mathbb{E}%
\left[  \left\vert \sigma(r)\right\vert ^{\rho}\right]  \right.  \right. \\
&  \left.  \left.  +\mathbb{E}\left[  \int_{%
\mathbb{R}
_{0}}\left\vert \mathcal{X}(r,e)\right\vert ^{\rho}\nu(de)\right]
+\mathbb{E}\left[  \left(  \int_{%
\mathbb{R}
_{0}}\left\vert \mathcal{X}(r,e)\right\vert ^{2}\nu(de)\right)  ^{\frac{\rho
}{2}}\right]  \right\}  dr\right]  .
\end{align*}
}
\end{lemma}

{\normalsize Now, define the random fields
\begin{align}
F^{\eta}_{\alpha}(t):=X^{\pi+\alpha\eta}(t)-X^{\pi}(t)
\end{align}
and
\begin{align}
\mathbf{F}^{\eta}_{\alpha}(t):=\mathbf{X}^{\pi+\alpha\eta}(t)-\mathbf{X}^{\pi
}(t)=\Big(\int_{-\delta}^{0} F_{\alpha}^{\eta}(t+r)\mu_{1}(dr),\ldots,
\int_{-\delta}^{0} F_{\alpha}^{\eta}(t+r)\mu_{N}(dr) \Big)
\end{align}
}

\begin{lemma}
{\normalsize Let $T\in(0,\infty)$. There exists a constant $C=C_{T}>0$,
independent of $\pi, \eta$, such that
\begin{align}
\label{convergenceResultX}\mathbb{E}[\sup_{0\leq v\leq t}|\mathbf{F}_{\alpha
}^{\pi,\eta}(v)|^{2}]\leq C\parallel\eta\parallel^{2}_{L^{2}(\Omega
\times[0,T])}\alpha^{2}.
\end{align}
whenever $t\leq T$. }

{\normalsize Moreover, there exists a measurable version of the map
\begin{align}
(t,\alpha,\omega)\mapsto\mathbf{F}_{\alpha}(t,\omega)
\end{align}
such that for a.e. $\omega$, $\mathbf{F}_{\alpha}(t)\rightarrow0$ as
$\alpha\rightarrow0$ for every $t\in[0,\infty)$. }
\end{lemma}

{\normalsize \begin{proof}The proof follows that of Lemma 4.2 in \cite{DMOE} closely. Define
\begin{align}
\beta_\alpha(t):= \mathbb E[\sup_{-\delta\leq v\leq t}| F_\alpha^{\eta}(v)|^2].
\end{align}Observe that using Jensen's inequality, we find that
\begin{align*}
\mathbb E[\sup_{0\leq v\leq t}|\mathbf F_\alpha^{\eta}(v)|^2]&=\mathbb E\Big[\sup_{0\leq v\leq t}\sum_{i=1}^N\Big|\int_{-\delta}^0 F_\alpha^{\eta}(v+r)\mu_i(dr)\Big|^2\Big]\\
&\leq  \mathbb E\Big[\sup_{0\leq v\leq t}\sum_{i=1}^N|\mu_i|\int_{-\delta}^0 |F_\alpha^{\eta}(v+r)|^2\mu_i(dr)\Big]\\
&\leq \mathbb E\Big[\sup_{0\leq v\leq t}\sum_{i=1}^N|\mu|^2\sup_{-\delta\leq r\leq 0}|F_\alpha^{\eta}(v+r)|^2\Big]\\
&\leq |\mu|^2 \beta_\alpha(t)
\end{align*}where $|\mu|:=\sum_{i=1}^N|\mu_i|:=\sum_{i=1}^N\mu_i[-\delta,0]$.
Since $\nabla b,\nabla \sigma,\nabla \gamma$ are bounded, $b,\sigma$ and $\gamma$ are Lipshitz in the variable $\mathbf x,\mathbf m, u$. Now using the integral representation of $X^{\pi+\alpha\eta}$ and $X^\pi$, Kunita's inequality, and finally the Lipshitz conditions on $b,\sigma$ and $\gamma$, we find that
\begin{align*}
\beta_\alpha(t)
&\leq C_{2,T}\mathbb E\Big[ \int_0^t |b(s,\pi+\alpha\eta)-b(s,\pi)|^2 + |\sigma(s,\pi+\alpha\eta)-\sigma(s,\pi)|^2\\
&\quad + \int_{\mathbb R_0}|\gamma(s,\pi+\alpha\eta,e)-\gamma(s,\pi,e)|^2 \nu(de)ds   \Big]\\
&\leq C_{2,T}(D_0^2+ \parallel D\parallel^2_{L^2(\nu)})\mathbb E\Big[ \int_0^t |\mathbf F_\alpha(s)|^2+ \alpha^2|\eta(s)|^2 ds\Big]\\
&\leq C_{2,T}(D_0^2+ \parallel D\parallel^2_{L^2(\nu)})\Big(\int_0^t\beta_\alpha(s)ds+ \alpha^2\parallel\eta\parallel^2_{L^2(\Omega\times[0,T])}\Big).
\end{align*}
Now (\ref{convergenceResultX}) holds by Gronwall's inequality.
The second part of the lemma follows by the same argument as in \cite{DMOE}.
\end{proof}
Now, fix $\pi$ and define
\begin{align*}
A^{\eta}_{\alpha}(t)  &  :=\frac{X^{\pi+\alpha\eta}(t)-X^{\pi}(t)}{\alpha
}-\mathcal{X}^{\pi,\eta}(t)\\
\mathbf{A}^{\eta}_{\alpha}(t)  &  :=\frac{\mathbf{X}^{\pi}(t)-\mathbf{X}%
^{\pi+\alpha\eta}(t)}{\alpha}-\bm{\mathcal X}^{\pi,\eta}(t)\\
&  =\Big(\int_{-\delta}^{0} A^{\eta}_{\alpha}(t+r)\mu_{1}(dr), \ldots,
\int_{-\delta}^{0} A^{\eta}_{\alpha}(t+r)\mu_{N}(dr) \Big)
\end{align*}
for each $\eta$. }

\begin{lemma}
{\normalsize \label{lemma:Jdifferentable1} For each $t\in(0,\infty)$, it holds
that
\begin{align*}
\label{x}\theta_{\alpha}(t):=  &  \mathbb{E}[\sup_{0\leq v\leq t}|A_{\alpha
}^{\eta}(v)|^{2}] \rightarrow0,\\
&  \mathbb{E}[\sup_{0\leq v\leq t}|\mathbf{A}_{\alpha}^{\eta}(v)|^{2}]
\rightarrow0
\end{align*}
as $\alpha\rightarrow0$. }
\end{lemma}

{\normalsize \begin{proof}
Similarly as in the previous proof, we find that
\begin{align*}
\mathbb E[\sup_{0\leq v\leq t}|\mathbf A_\alpha^{\eta}(v)|^2]
&\leq N|\mu|^2 \theta_s(t).
\end{align*}
The rest of the proof follows the exact same steps as the proof of Lemma 4.3 in \cite{DMOE} and is therefore omitted.
\end{proof}
}

\subsubsection{{\protect\normalsize Differentiability of the backward state
process}}

{\normalsize We will assume that the following convergence results hold for
all $t\geq0$
\begin{equation}
\mathbb{E}\left[  \underset{0\leq r\leq t}{\sup}\left\vert \frac{Y^{\pi
+\alpha\eta}(r)-Y^{\pi}(r)}{\alpha}-\mathcal{Y}(r)\right\vert ^{2}\right]
\rightarrow0, \label{y}%
\end{equation}%
\begin{equation}
\mathbb{E}\left[  \int_{0}^{T}\left\vert \frac{Z^{\pi+\alpha\eta}(t)-Z^{\pi
}(t)}{\alpha}-\mathcal{Z}(t)\right\vert ^{2}dt\right]  \rightarrow0,,
\label{z}%
\end{equation}
and
\begin{equation}
\text{ }\mathbb{E}\left[  \int_{0}^{T}\int_{%
\mathbb{R}
_{0}}\left\vert \frac{K^{\pi+\alpha\eta}(t,e)-K^{\pi}(t,e)}{\alpha
}-\mathcal{K}(t,e)\right\vert ^{2}\nu(de)dt\right]  \rightarrow0 \label{k}%
\end{equation}
as $\alpha\rightarrow0.$ In particular, $\mathcal{Y}$, $\mathcal{Z}$ and
$\mathcal{K}$ are the solutions of (\ref{eq:DerivativeYZK}). We refer to
\cite{Elin} for more details.}

\subsection{{\protect\normalsize Differentiability of the performance
functional}}

\begin{lemma}
[\textbf{Differentiability of }$J$]Suppose $\pi,\mathcal{\eta}\in
\mathcal{A}_{\mathbb{G}}$ with $\mathcal{\eta}$ bounded. Suppose there exist
an interval $I\subset%
\mathbb{R}
$ with $0\in I$, such that the perturbation $\pi+s\mathcal{\eta}$ is in
$\mathcal{A}_{\mathbb{G}}$ for each $s\in I.$ Then the function $s\mapsto
\dfrac{d}{ds}J(\pi+s\mathcal{\eta})$ has a (possibly one-sided) derivative at
$0$ with {\small
\begin{align*}
&  \frac{d}{ds}J(\pi+s\eta)\mid_{s=0}\\
&  =\mathbb{E}\left[  \int_{0}^{\infty}\nabla f(t,\mathbf{X}^{\pi}%
(t),\pi)\cdot(\mathcal{X}^{\pi}(t),\eta(t))^{T}dt+h^{\prime}(Y^{\pi}%
(0))\cdot\mathcal{Y}\text{ }^{\pi,\eta}(0)\right]  .
\end{align*}
}
\end{lemma}

\begin{proof}
Recall that {\small
\[
J(\pi)=\mathbb{E}\left[  \int_{0}^{\infty}f(t,\mathbf{X}^{\pi}(t),\pi
)dt+h(Y^{\pi}(0))\right]  .
\]
}
Let {\small
\[
J_{n}(\pi)=\mathbb{E}\left[  \int_{0}^{n}f(t,\mathbf{X}^{\pi}(t),\pi
)dt+h(Y^{\pi}(0))\right]  .
\]
}
We show that {\small
\begin{align*}
&  \frac{d}{ds}J_{n}(\pi+s\eta)\mid_{s=0}\\
&  =\mathbb{E}\left[  \int_{0}^{n}\nabla f(t,\mathbf{X}^{\pi}(t),\pi
)\cdot(\mathcal{X}^{\pi}(t),\eta(t))^{T}dt+h^{\prime}(Y^{\pi}(0))\cdot
\mathcal{Y}\text{ }^{\pi,\eta}(0)\right]  .
\end{align*}
}
Since $\nabla f$ and $\eta$ are bounded and $\mathcal{X}$ is finite, we can
use the dominated convergence theorem to get{\small
\[
\frac{d}{ds}J(\pi+s\eta)\mid_{s=0}=\underset{n\rightarrow\infty}{\lim}\frac
{d}{ds}J_{n}(\pi+s\eta)\mid_{s=0}.
\]
}
Now proving{\small
\[
\frac{d}{ds}\mathbb{E}\left[  h(Y^{\pi+s\eta}(0))\right]  \mid_{s=0}%
=\mathbb{E[}h^{\prime}(Y^{\pi}(0))\cdot\mathcal{Y}\text{ }^{\pi,\eta}(0)]
\]
}
{\small
\begin{align*}
& \left\vert \mathbb{E}\left[  \frac{1}{s}h(Y^{\pi+s\eta}(0))-h(Y^{\pi
}(0))-h^{\prime}(Y^{\pi}(0))\cdot\mathcal{Y}\text{ }^{\pi,\eta}(0)\right]
\right\vert \\
& =\left\vert \mathbb{E}\left[  \frac{1}{s}\int_{0}^{1}h^{\prime}(Y^{\pi
}(0)+\lambda(Y^{\pi+s\eta}(0))\cdot(Y^{\pi+s\eta}(0)-Y^{\pi}(0))d\lambda
\right.  \right.  \\
& \left.  \left.  -h^{\prime}(Y^{\pi}(0))\cdot\mathcal{Y}\text{ }^{\pi,\eta
}(0)\right]  \right\vert \\
& =\left\vert \mathbb{E}\left[  \int_{0}^{1}h^{\prime}(Y^{\pi}(0)+\lambda
(Y^{\pi+s\eta}(0)-Y^{\pi}(0))\cdot\left(  \frac{Y^{\pi+s\eta}(0)-Y^{\pi}%
(0)}{s}-\mathcal{Y}\text{ }^{\pi,\eta}(0)\right)  \right.  \right.  \\
& \left.  \left.  +\left\{  h^{\prime}(Y^{\pi}(0)+\lambda(Y^{\pi+s\eta
}(0)-Y^{\pi}(0))-h^{\prime}(Y^{\pi}(0))\cdot\mathcal{Y}\text{ }^{\pi,\eta
}(0)\right\}  d\lambda\right]  \right\vert \\
& \leq\int_{0}^{1}D\mathbb{E}\left[  \left\vert \frac{Y^{\pi+s\eta}(0)-Y^{\pi
}(0)}{s}-\mathcal{Y}\text{ }^{\pi,\eta}(0)\right\vert \right]  d\lambda\\
& +\int_{0}^{1}L\mathbb{E}\left[  \left\vert \lambda(Y^{\pi+s\eta}(0)-Y^{\pi
}(0))\right\vert \cdot\left\vert \mathcal{Y}\text{ }^{\pi,\eta}(0)\right\vert
\right]  d\lambda\\
& \rightarrow0\text{ as }s\rightarrow0
\end{align*}
}
we obtain the last estimation by, assumption of ${\normalsize \mathcal{Y}}$
$\left(  \ref{y}\right)  $ and apply Cauchy-Schwartz inequality, we obtain
\begin{align*}
& \mathbb{E}\left[  \left\vert Y^{\pi+s\eta}(0)-Y^{\pi}(0)\right\vert
\cdot\left\vert \mathcal{Y}\text{ }^{\pi,\eta}(0)\right\vert \right]  \\
& \leq\mathbb{E}\left[  \left\vert Y^{\pi+s\eta}(0)-Y^{\pi}(0)\right\vert
^{2}\right]  ^{\frac{1}{2}}\mathbb{E}\left[  \left\vert \mathcal{Y}\text{
}^{\pi,\eta}(0)\right\vert ^{2}\right]  ^{\frac{1}{2}}\\
& \rightarrow0\text{ as }s\rightarrow0.
\end{align*}
\end{proof}

\begin{lemma}
{\normalsize \label{lemma:JdifferentialH} [\textbf{Differentiability of $J$ in
terms of the Hamiltonian}] Let $\pi,\eta\in\mathcal{A}_{\mathbb{G}}$ with
$\eta$ bounded. Let $X,Y,Z,K,p,q,r,\lambda$ be the state and corresponding to
$\pi$ adjoint processes , and $\mathcal{X},\mathcal{Y},\mathcal{Z}%
,\mathcal{K},$ be derivative processes corresponding to $\pi,\eta$. Suppose
there exists adjoint processes $p,q,r$ corresponding to $\pi$ and that
\begin{align}
&  \underset{T\rightarrow\infty}{\lim}\mathbb{E}[\mathbf{\ }{p}(T)\mathcal{X}%
(T)]=0,\label{eq:p_theta_condition}\\
&  \underset{T\rightarrow\infty}{\lim}\mathbb{E}[{\lambda}(T)\mathcal{Y}%
(T)]=0. \label{eq:lambda_phi_condition}%
\end{align}
Then
\[
\dfrac{d}{ds}J(\pi+s\eta)\mid_{_{s=0}}=\mathbb{E}\left[  \int_{0}^{\infty
}\frac{\partial H}{\partial\pi}\left(  t,\pi\right)  \eta(t)dt\right]  .
\]
}
\end{lemma}

{\normalsize \begin{proof}
For fixed $T\geq 0$, define a sequence of stopping times, as follows
\small\begin{align}\begin{split}
\tau_n(\cdot):=T\wedge\inf\bigg\{&t\geq 0:\int_0^t\Big(  p(t) \nabla \sigma(t,\pi)\cdot \Big(  \mathbb{X}(t),\mathbb{E}[\mathbb{X}(t)],\eta(t)\Big)^{\mathsf{T}} +\mathcal{X}(t)q(t)    \Big)^2 ds\\
&+ \int_0^t \int_{\mathbb R_0}\Big(  r(s,e)\nabla\gamma(s,\pi,e)\cdot\left(  \mathbb{X}(s),\mathbb{E}[\mathbb{X}(s)],\eta(s)\right)^{\mathsf{T}}\\
& + p(s)\nabla\gamma(s,\pi,e)\cdot\left(  \mathbb{X}(s),\mathbb{E}[\mathbb{X}(s)],\eta(s)\right)^{\mathsf{T}} + \mathcal{X}(s)r(s,\pi,e)   \Big)^2\nu(de)ds\\
&+\int_0^t\Big( \mathcal{Y}(s)\dfrac{\partial H}{\partial z}(s,\pi)+ \lambda(s)\mathcal{Z}(s)\Big)^2 ds\\
&\int_0^{t}\int_{\mathbb R_0}\Big(\mathcal{Y}(s)\nabla_k H(s,e)     +\lambda(s)\mathcal{K}(s,e)+ \mathcal{K}(s,e)\nabla_k H(s,e)\Big)^2\nu(de) ds\geq n\bigg\} ,\quad n\in\mathbb N
\end{split}\end{align}\normalsize
now it clearly holds that $\tau_n\rightarrow T$ $P$-a.s.
Observe  that, with a slight abuse of notation(in particular, we write $\nabla b$ etc. both when we consider
it as a Fréchet derivative  with respect the spacial variables from (\ref{EQ27}) and when considering it as the gradient with respect to all the spacial variables of the Hamiltonian, $H$), it holds that
\begin{align}\label{eq:nablaH}
\nabla H-\lambda\nabla g-\nabla f=p\nabla b+q\nabla \sigma+\int_{\mathbb R_0}r\nabla \gamma \,d\nu.
\end{align}(This can be shown using the Chain rule for the Fréchet derivative.)
By It\^{o}'s formula, we can compute that
\small\begin{align}\label{eq:p_theta}\begin{split}
p(\tau_n)\mathcal{X}(\tau_n)&=\int_0^{\tau_n} \Big(p(t)\nabla b(t,\pi)+q(t)\nabla \sigma(t,\pi)  + \int_{\mathbb{R}_0}r(t,e)\nabla\gamma(t,\pi,e) \nu(de) \Big)\\
&\quad   \cdot\Big(  \mathbb{X}(t),\mathbb E[\mathbb{X}(t)],\eta(t)\Big)^{\mathsf{T}}- \mathcal{X}(t) \mathbb E[\Upsilon(t)|\mathfrak F_t]\, dt\\
&+\int_0^{\tau_n} p(t) \nabla \sigma(t,\pi)\cdot \Big(  \mathbb{X}(t),\mathbb E[\mathbb{X}(t)],\eta(t)\Big)^{\mathsf{T}} +\mathcal{X}(t)q(t)dB(t)\\
&+\int_0^{\tau_n}\int_{\mathbb R_0}r(t,e)\nabla\gamma(t,\pi,e)\cdot\left(  \mathbb{X}(t),\mathbb E[\mathbb{X}(t)],\eta(t)\right)^{\mathsf{T}}\\
& + p(t)\nabla\gamma(t,\pi,e)\cdot\left(  \mathbb{X}(t),\mathbb E[\mathbb{X}(t)],\eta(t)\right)^{\mathsf{T}} + \mathcal{X}(t)r(t,\pi,e)\tilde N(dt,de)
\end{split}\end{align}\normalsize
The stochastic integral parts have zero mean by definition of the stopping time, and we recall that $\mathcal{X}(0)=0$.
Observe that since we have required that all solutions of the  state and adjoint equations belongs to the spaces $L^2(\Omega\times[0,\infty))$ or $L^2(\Omega\times[0,\infty)\times \mathbb R_0)$, and that the gradient of the coefficient functionals are bounded,   it holds that
\small\begin{align}\label{ineq:intCondNMP}\begin{split}
\mathbb E\Big[ \int_0^{\infty} &\Big|\Big(p(t)\nabla b(t,\pi)+q(t)\nabla \sigma(t,\pi)  + \int_{\mathbb{R}_0}r(t,e)\nabla\gamma(t,\pi,e) \nu(de) \Big)\\
&\quad   \cdot\Big( \mathbb{X}(t),\mathbb E[\mathbb{X}(t)],\eta(t)\Big)^{\mathsf{T}}\Big|+ |\mathcal{X}(t)\Upsilon(t)|\,dt\Big]<\infty.
\end{split}
\end{align}\normalsize
Now,
\begin{align}\small\label{eq:EpTau}\begin{split}
\mathbb E[p(T)&\mathcal{X}(T)]=\lim_{n\rightarrow \infty}\mathbb E[p(\tau_n)\mathcal{X}(\tau_n)]\\
&=\mathbb E\Big[\int_0^{\tau_n} \Big(p(t)\nabla b(t,\pi)+q(t)\nabla \sigma(t,\pi)  + \int_{\mathbb{R}_0}r(t,e)\nabla\gamma(t,\pi,e) \nu(de) \Big)\\
&\quad   \cdot\Big( \mathbb{X}(t),\mathbb E[\mathbb{X}(t)],\eta(t)\Big)^{\mathsf{T}}- \mathcal{X}(t) \mathbb E[\Upsilon(t)|\mathfrak F_t]\, dt\Big]\\
&=\lim_{n\rightarrow \infty}\mathbb E\Big[ \int_0^{\tau_n} \bigg(\nabla H(t,\pi)-\lambda(t)\nabla g(t,\pi)-\nabla f(t,\pi)\bigg)\\
&\quad\cdot\Big(  \mathbb{X}(t),\mathbb{E}[\mathbb{X}(t)],\mathcal{Y}(t),\mathbb{E}[\mathcal{Y}(t)],\mathcal{Z}(t),\mathcal{K}(t),\eta(t)\Big)^{\mathsf T} - \mathcal{X}(t) \Upsilon(t)dt\Big]\\
&=\mathbb E\Big[ \int_0^{\tau_n} \bigg(\nabla H(t,\pi)-\lambda(t)\nabla g(t,\pi)-\nabla f(t,\pi)\bigg)\\
&\quad\cdot\Big(  \mathbb{X}(t),\mathbb{E}[\mathbb{X}(t)],\mathcal{Y}(t),\mathbb{E}[\mathcal{Y}(t)],\mathcal{Z}(t),\mathcal{K}(t),\eta(t)\Big)^{\mathsf T} - \mathcal{X}(t) \Upsilon(t)dt\Big].
\end{split}
\normalsize\end{align}
In the first and last equality, we have used Lebesgue's dominated convergence theorem and that the integrand is dominated by the integrable random variable in (\ref{ineq:intCondNMP}). In the second equality, we have used the integral representation (\ref{eq:p_theta}) of $p(\tau_n)\mathcal{X}(\tau_n)$ and that the stochastic integrals have zero mean by definition of the stopping times $\tau_n$, and in the third equality, we have used  (\ref{eq:nablaH}).
From the assumption (\ref{eq:lambda_phi_condition}), and again using the fact that the integrands are dominated by the integrable random variable in (\ref{ineq:intCondNMP}), we find that\begin{align}\small\label{eq:Ept}\begin{split}
0=&\lim_{T\rightarrow\infty}\mathbb E[p(T)\mathcal{X}(T)]=\mathbb E\Big[ \int_0^{\infty} \bigg(\nabla H(t,\pi)-\lambda(t)\nabla g(t,\pi)-\nabla f(t,\pi)\bigg)\\
&\quad\cdot\Big(  \mathbb{X}(t),\mathbb{E}[\mathbb{X}(t)],\mathcal{Y}(t),\mathbb{E}[\mathcal{Y}(t)],\mathcal{Z}(t),\mathcal{K}(t),\eta(t)\Big)^{\mathsf T} - \mathcal{X}(t) \Upsilon(t)\,dt\Big].
\end{split}\normalsize
\end{align}
Similarly, using It\^{o}'s formula, we compute that \begin{align}\label{eq:lambda_phi}\small\begin{split}
\lambda(\tau_n)\mathcal{Y}(\tau_n)&-\lambda(0)\mathcal{Y}(0)= \int_0^{\tau_n}\Big\{\mathcal{Y}(t)\bigg(\frac{\partial H}{\partial y}(t,\pi)+ \mathbb E\Big[\frac{\partial H}{\partial n}(t,\pi)\Big]\bigg)\\
&+\lambda(t)\Big(-\nabla g(t,\pi)\Big)\cdot\Big(  \mathbb{X}(t),\mathbb{E}[\mathbb{X}(t)],\mathcal{Y}(t),\mathbb{E}[\mathcal{Y}(t)],\mathcal{Z}(t),\mathcal{K}(t),\eta(t)\Big)^{\mathsf T}\\
&+\mathcal{Z}(t)\frac{\partial H}{\partial z}(t,\pi)+\int_{\mathbb R_0}\mathcal{K}(t,e)\nabla _{k}H(t,e)\nu (de)\Big\}dt\\
&+\int_0^{\tau_n} \mathcal{Y}(t)\dfrac{\partial H}{\partial z}(t,\pi)+ \lambda(t)\mathcal{Z}(t) dB(t)\\
&+\int_0^{\tau_n}\int_{\mathbb R_0}\mathcal{Y}(t)\nabla _{k}H(t,e)     +\lambda(t)\mathcal{K}(t,e)+ \mathcal{K}(t,e)\nabla _{k}H(t,e)  \tilde N(dt, de).
\end{split}\normalsize\end{align}
We recall that  $\lambda(0)=h'(Y(0))$. Then proceeding as above, we find that
\begin{align}\label{lim:lambda_phi}
\begin{split}
0&=\lim_{T\rightarrow \infty}\mathbb E\Big[ \lambda(T)\mathcal{Y}(T)  \Big]\\
&=\mathbb E\Big[ h'(Y(0))\mathcal{Y}(0)  \Big]
+E\Big[    \int_0^\infty\Big\{\mathcal{Y}(t)\bigg(\frac{\partial H}{\partial y}(t,\pi)+ \mathbb E\Big[\frac{\partial H}{\partial n}(t,\pi)\Big]\bigg)\\
&+\lambda(t)\Big(-\nabla g(t,\pi)\Big)\cdot\Big(  \mathbb{X}(t),\mathbb{E}[\mathbb{X}(t)],\mathcal{Y}(t),\mathbb{E}[\mathcal{Y}(t)],\mathcal{Z}(t),\mathcal{K}(t),\eta(t)\Big)^{\mathsf T}\\
&+\mathcal{Z}(t)\frac{\partial H}{\partial z}(t,\pi)+\int_{\mathbb R_0}\mathcal{K}(t,e)\nabla _{k}H(t,e)\nu (de)\Big\} dt\Big]
\end{split}
\end{align}
Now, combining Lemma 4.4 with the equations (\ref{eq:Ept}) and (\ref{lim:lambda_phi}) yields
\small
\begin{align}
\frac{d}{d\alpha}\mathbb E &\Big[J(\pi+\alpha\eta)  \Big]\Big|_{\alpha=0}=\mathbb{E}\Big[
h^{\prime}(Y(0))\,\mathcal{Y}(0)\nonumber\\
&   +\int_{0}^{\infty}\nabla f(t,\pi)\cdot\Big(\mathbb{X}(t), \mathbb E[\mathbb{X}(t)], \mathcal{Y}(t), \mathbb E[\mathcal{Y}(t)], \mathcal{Z}(t), \mathcal{K}(t,\cdot), \eta(t)\Big)^{\mathsf T}dt\Big]\nonumber\\
&= \mathbb E\Big[\int_0^\infty \lambda(t)\Big(\nabla g(t,\pi)\Big)\cdot\Big(  \mathbb{X}(t),\mathbb{E}[\mathbb{X}(t)],\mathcal{Y}(t),\mathbb{E}[\mathcal{Y}(t)],\mathcal{Z}(t),\mathcal{K}(t),\eta(t)\Big)^{\mathsf T}\nonumber\\
&-\Big\{\mathcal{Y}(t)\bigg(\frac{\partial H}{\partial y}(t,\pi)+ \mathbb E\Big[\frac{\partial H}{\partial n}(t,\pi)\Big]\bigg)+ \mathcal{Z}(t)\frac{\partial H}{\partial z}(t,\pi)+\int_{\mathbb R_0}\nabla_k H(t,\pi)\mathcal{K}(t,e)\nu (de)\Big\}\, dt\Big]\nonumber\\
&+\mathbb E\Big[\int_0^\infty \bigg(\nabla H(t,\pi)-\lambda(t)\nabla g(t,\pi)\bigg)\nonumber\\
&\quad \cdot\bigg(  \mathbb{X}(t),\mathbb{E}[\mathbb{X}(t)],\mathcal{Y}(t),\mathbb{E}[\mathcal{Y}(t)],\mathcal{Z}(t),\mathcal{K}(t),\eta(t)\bigg)^{\mathsf T}- \mathcal{X}(t) \Upsilon(t)\, dt\Big]\nonumber\\
&=\mathbb E\Big[\int_0^\infty \nabla H(t,\pi)\cdot\bigg(  \mathbb{X}(t),\mathbb{E}[\mathbb{X}(t)],\mathcal{Y}(t),\mathbb{E}[\mathcal{Y}(t)],\mathcal{Z}(t),\mathcal{K}(t),\eta(t)\bigg)^{\mathsf T}\nonumber\\
&-\Big\{ \sum_{i=0}^{N-1}\int_{-\delta}^0 \mathcal{X}(t)\Big( \frac{\partial H}{\partial x_i}(t-s,\pi) +\mathbb E\Big[\frac{\partial H}{\partial m_i}(t-s,\pi) \Big]\Big)\mu_i(ds)\nonumber\\
&+ \mathcal{Y}(t)\bigg(\frac{\partial H}{\partial y}(t,\pi)+ \mathbb E\Big[\frac{\partial H}{\partial n}(t,\pi)\Big]\bigg)+ \mathcal{Z}(t)\frac{\partial H}{\partial z}(t,\pi)+\int_{\mathbb R_0}\mathcal{K}(t,e)\nabla_k H(t,e)\nu (de)  \Big\} dt\Big]\nonumber\\
&=\mathbb E\Big[\sum_{i=1}^N\Big\{\frac{\partial H}{\partial x_i}(t,\pi)\cdot\int_{-\delta}^0 \mathcal{X}(t+s)\mu_i(ds)+\frac{\partial H}{\partial m_i}(t,\pi)\cdot\mathbb E\big[\int_{-\delta}^0 \mathcal{X}(t+s)\mu_i(ds)\big]   \Big\}\label{NMPproof_tmp1}\\
&\quad+ \frac{\partial H}{\partial y}(t,\pi)\mathcal{Y}(t)+\frac{\partial H}{\partial n}\mathbb E[\mathcal{Y}(t)]+ \frac{\partial H}{\partial z}(t,\pi)\mathcal{Z}(t)+ \int_{\mathbb R_0}\nabla_k H(t,\pi)(e)\mathcal{K}(t,e)\nu(de)\label{NMPproof_tmp2}\\
&\quad+\frac{\partial H}{\partial u}(t,\pi)\eta(t)\\
&\quad-\Big\{ \sum_{i=0}^{N-1}\int_{-\delta}^0 \mathcal{X}(t)\Big( \frac{\partial H}{\partial x_i}(t-s,\pi) +\mathbb E\Big[\frac{\partial H}{\partial m_i}(t-s,\pi) \Big]\Big)\mu_i(ds)\label{NMPproof_tmp3}\\
&\quad+ \mathcal{Y}(t)\bigg(\frac{\partial H}{\partial y}(t,\pi)+ \mathbb E\Big[\frac{\partial H}{\partial n}(t,\pi)\Big]\bigg)+ \mathcal{Z}(t)\frac{\partial H}{\partial z}(t,\pi)+\int_{\mathbb R_0}\nabla_k H(t,e)\mathcal{K}(t,e)\nu (de)  \Big\} dt\Big]\label{NMPproof_tmp4}\\
&=\mathbb E \Big[\int_0^\infty\frac{\partial H}{\partial \pi}(t,\pi) \eta(t)dt\Big],\nonumber
\end{align}
\normalsize which is what we wanted to prove.
In order to see why the last equality holds, observe that one may use Fubini's theorem to show that the sum of the lines (\ref{NMPproof_tmp2}) and (\ref{NMPproof_tmp4}) is $0$. Also, we have that the sum of the  lines (\ref{NMPproof_tmp1}) and (\ref{NMPproof_tmp3}) is $0$. To see why the latter holds, recall that $\mathcal{X}(r)=0,$ when $r<0$, and perform the change of variable $r=t-s$ in the $dt$-integral to observe  that
\small\begin{equation}
\begin{split}
\mathbb E\Big[&\int_0^\infty  \int_{-\delta}^0 \mathcal{X}(t)\Big( \frac{\partial H}{\partial x_i}(t-s,\pi) +\mathbb E\Big[\frac{\partial H}{\partial m_i}(t-s,\pi) \Big]\Big)\mu_i(ds) dt\Big]\\
&=\mathbb E\Big[\int_{-\delta}^0\int_0^\infty  \mathcal{X}(t)\Big( \frac{\partial H}{\partial x_i}(t-s,\pi) +\mathbb E\Big[\frac{\partial H}{\partial m_i}(t-s,\pi) \Big]\Big)dt\,\mu_i(ds) \Big]\\
&=\mathbb E\Big[\int_{-\delta}^0\int_s^\infty \mathcal{X}(t)\Big( \frac{\partial H}{\partial x_i}(t-s,\pi) +\mathbb E\Big[\frac{\partial H}{\partial m_i}(t-s,\pi) \Big]\Big)dt\,\mu_i(ds) \Big]\\
&=\mathbb E\Big[\int_{-\delta}^0\int_0^\infty \mathcal{X}(r+s)\Big( \frac{\partial H}{\partial x_i}(r,\pi) +\mathbb E\Big[\frac{\partial H}{\partial m_i}(r,\pi) \Big]\Big)dt\,\mu_i(ds) \Big]\\
\color{red}
&=\mathbb E\Big[\int_0^\infty \int_{-\delta}^0\mathcal{X}(t+s)\Big( \frac{\partial H}{\partial x_i}(t,\pi) +\mathbb E\Big[\frac{\partial H}{\partial m_i}(t,\pi) \Big]\Big)\mu_i(ds)\,dt \Big]\\
&=\mathbb E\Big[\int_{0}^\infty \int_{-\delta}^0 \mathcal{X}(t+s)\Big( \frac{\partial H}{\partial x_i}(t,\pi) +\mathbb E\Big[\frac{\partial H}{\partial m_i}(t,\pi) \Big]\Big)\,\mu_i(ds)dt \Big]
\end{split}
\end{equation}\normalsize
\end{proof}
}

\begin{theorem}
[Necessary maximum principle]{\normalsize Under the assumptions of Lemma 4.5,
we can prove the equivalence between: }

{\normalsize $(i)$ For each bounded $\eta\in\mathcal{A}_{\mathbb{G}},$%
\[
0=\dfrac{d}{ds}J(\pi+s\eta)\mid_{_{s=0}}=\mathbb{E}\left[  \int_{0}^{\infty
}\frac{\partial H}{\partial\pi}\left(  t,\pi\right)  \eta(t)dt\right]
\]
}

{\normalsize $(ii)$ For each $t\in\lbrack0,\infty),$%
\[
\mathbb{E}\left[  \frac{\partial H}{\partial\pi}\left(  t,\pi\right)
\Big \vert\mathfrak{G}_{t}\right]  _{\pi=\pi^{\ast}(t)}=0\text{ a.s.}%
\]
}
\end{theorem}

\begin{proof}
{\normalsize Using Lemma $4.4$, the proof is similar to the proof of Theorem
$4.1$ in \cite{AO}. }
\end{proof}

\subsection{{\protect\normalsize Sufficient maximum principle}}

\begin{theorem}
[Sufficient maximum principle]{\normalsize Let $\pi\in\mathcal{A}_{\mathbb{G}%
}$ with corresponding solutions $X(t),Y(t),Z(t),K(t,\cdot),p(t),q(t),r(t,\cdot
),\lambda(t)$. Assume the following conditions hold: }

\begin{description}
\item[$(i)$] {\normalsize
\begin{align}
& \nonumber\\
\mathbb{E}\left[  H\left(  t,\pi\right)  \Big\vert\mathfrak{G}_{t}\right]   &
=\underset{v\in U}{\sup}\mathbb{E}\left[  H\left(  t,v\right)
\Big\vert\mathfrak{G}_{t}\right]  , \label{EQ34}%
\end{align}
}
\end{description}

{\normalsize for all $t\in\lbrack0,\infty)$ a.s. }

\begin{description}
\item[$(ii)$] {\normalsize Transversality conditions%
\begin{equation}
\underset{T\rightarrow\infty}{\lim}\mathbb{E}\mathbf{\ }\left[  \hat
{p}(T)\left(  \hat{X}(T)-X(T)\right)  \right]  \leq0 \label{EQ35}%
\end{equation}%
\begin{equation}
\underset{T\rightarrow\infty}{\lim}\mathbb{E}\mathbf{\ }\left[  \hat{\lambda
}(T)\left(  \hat{Y}(T)-Y(T)\right)  \right]  \geq0\text{.} \label{EQ36}%
\end{equation}
}
\end{description}

{\normalsize Then $\pi$ is an optimal control for the problem $\left(
\ref{EQ30}\right)  .$ }
\end{theorem}

{\normalsize \begin{proof}
The proof is similar to the proof of Theorem 2.5 but with infinite time horizon and Theorem $3.1$ in \cite{AO}.
\end{proof}
}

\section{{\protect\normalsize Optimal consumption with respect to recursive
utility}}

{\normalsize Suppose now that the state equation is a cash flow on the form
\begin{equation}
\left\{
\begin{array}
[c]{l}%
dX(t)=\left[  b_{0}(t,\mathbb{E}[X(t)])-\pi(t)\right]  dt+\sigma
(t,\mathbf{X}(t),\mathbb{E[}X(t)],\pi(t))dB(t)\\
\quad\quad\quad\quad+\int_{\mathbb{R}_{0}}\mathcal{X}(t,\mathbf{X}%
(t),\mathbb{E[}X(t)],\pi(t),e)\tilde{N}(dt,de),\quad t\in\lbrack0,\infty),\\
X(t)=X_{0}(t),\quad t\in\left[  -\delta,0\right]  ,
\end{array}
\right.  \label{SDE}%
\end{equation}
where the control $\pi(t)\geq0$ represents a consumption rate. The function
$b_{0}$ is assumed to be deterministic, in addition to the assumptions from
the previous sections. We want to consider an optimal recursive utility
problem similar to the one in \cite{AO}. See also \cite{DE}. For notational
conveniencen assume that }$\mu_{0}${\normalsize is the Dirac measure
concentrated at }${\normalsize 0.}$

{\normalsize Define the recursive utility process $Y(t)=Y^{\pi}(t),$ by the
BSDE in the unknown processes $(Y,Z,K)=(Y^{\pi},Z^{\pi},K^{\pi})$, by
\begin{equation}%
\begin{array}
[c]{l}%
dY(t)=-g(t,\mathbf{X}(t),\mathbb{E}[X(t)],Y(t),\mathbb{E[}Y(t)],\pi
(t),\omega)dt+Z(t)dB(t)\\
+\int_{\mathbb{R}_{0}}K(t,e,\omega)\tilde{N}(dt,de),\quad t\in\lbrack
0,\infty).
\end{array}
\label{BSDE}%
\end{equation}
}

{\normalsize We assume that equation $\left(  \ref{EQ28}\right)  $ satisfies
$\left(  \ref{decayp}\right)  $ and for all finite $T$ this is equivalent to%
\begin{equation}
Y(t)=E\left[  \left.  Y(T)+\int_{t}^{T}g(s,\mathbf{X}(s),\mathbb{E}%
[X(s)],Y(s),\mathbb{E[}Y(s)],\pi(s))ds\right\vert \mathfrak{F}_{t}\right]
;t\leq T\leq\infty.
\end{equation}
}

{\normalsize Notice that the function $b_{0}$ from the state equation $\left(
\ref{EQ27}\right)  $ depends only on $\mathbb{E}[X(t)]$ and on the control
$\pi(t)$, and that the driver $g$ is independent of $Z$. We have put no
further restrictions on the coefficient functionals so far. Let $f=0,h=0$ and
$h_{1}(y)=y$, in particular, we want to maximize the performance functional
\[
J(\pi):=Y^{\pi}(0)=\mathbb{E[}Y^{\pi}(0)\mathbb{]}.
\]
The admissible controls are assumed to be the càdlàg, $\mathfrak{G}_{t}%
$-adapted non-negative processes in $L^{2}(\Omega\times\lbrack0,T]).$ }

{\normalsize The adjoint processes $(p,q,r)=(p^{\pi},q^{\pi},r^{\pi})$ and
$\lambda=\lambda^{\pi}$ corresponding to $\pi$ are defined by }

{\normalsize
\begin{equation}
dp(t)=-\mathbb{E}[\Upsilon(t)|\mathfrak{F}_{t}]dt+q(t)dB(t)+\int
_{\mathbb{R}_{0}}r(t,e)\tilde{N}(dt,de), \label{example:adEq1}%
\end{equation}
}

{\normalsize with}%

\[%
\begin{array}
[c]{l}%
\Upsilon(t)=\frac{\partial}{\partial x_{0}}b_{0}(t)+\mathbb{E}\left[
\frac{\partial}{\partial m_{0}}b_{0}(t)\right] \\
+\sum_{i=0}^{2}\int_{-\delta}^{0}\left\{  \frac{\partial H}{\partial x_{i}%
}\Big(t-s,\pi\Big)+\mathbb{E}\Big[\frac{\partial H}{\partial m_{i}%
}\Big(t-s,\pi\Big)\Big]\right\}  \mu_{i}(ds)
\end{array}
\]

{\normalsize and}

{\normalsize
\begin{equation}
\left\{
\begin{array}
[c]{l}%
d\lambda(t)=\lambda(t)\Big(\frac{\partial}{\partial y}g(t,\pi)+\mathbb{E}%
\Big[\frac{\partial}{\partial n}g(t,\pi)\Big]\Big)dt,t\in\lbrack0,\infty)\\
\lambda(0)=h_{1}^{\prime}(Y(0))=1.
\end{array}
\right.  \label{Eqlamda}%
\end{equation}
}

{\normalsize We assume that equation $\left(  \ref{example:adEq1}\right)  $
satisfies the decay condition $\left(  \ref{decayp}\right)  .$ }

{\normalsize The Hamiltonian for the forward-backward system is
\begin{align}
H(t,\pi)  &  =\big(b_{0}(t,\mathbb{E}[x])-\pi\big)p+\sigma(t,\pi)q\nonumber\\
&  \quad+\int_{\mathbb{R}_{0}}\gamma(t,\pi,e)r(t,e)\nu(de)+g(t,\pi)\lambda
\end{align}
and
\begin{align}
\frac{\partial}{\partial\pi}H\big(t,\pi\big)  &  =-p(t)+\frac{\partial
}{\partial\pi}\Big(\sigma(t,\pi(t))q(t)\nonumber\\
&  +\int_{\mathbb{R}_{0}}\gamma(t,\pi(t),e)r(t,e)\nu(de)+g(t,\pi
(t))\lambda(t)\Big). \label{hamhu}%
\end{align}
Now, applying the necessary maximum principle to the expression above yields
the following: }

\begin{corollary}
{\normalsize Suppose that $\hat{\pi}(t)$ is an optimal control. Then {\small
\begin{equation}
\mathbb{E}[\hat{p}(t)|\mathfrak{G}_{t}]=\mathbb{E}\Big[\frac{\partial
}{\partial\pi}\Big(\sigma(t,\hat{\pi}(t))\hat{q}(t)+\int_{\mathbb{R}_{0}%
}\gamma(t,\hat{\pi}(t),e)\hat{r}(t,e)\nu(de)+g(t,\hat{\pi}(t))\hat{\lambda
}(t)\Big)\big|\mathfrak{G}_{t}\Big]. \label{fin}%
\end{equation}
} }
\end{corollary}

{\normalsize {\small We see that if we can put additional conditions on the
forward-backward system such that $\hat{q}=0,\hat{r}=0,$ $\hat{\lambda}$ is
deterministic with $\hat{\lambda}>0$ and that }$\mathfrak{G}_{t}%
=\mathfrak{F}_{t},$\ then $(\ref{fin})$\ reduces to\
\begin{equation}
\frac{\hat{p}(t)}{\hat{\lambda}(t)}=\frac{\partial}{\partial\pi}g(t,\hat{\pi
}). \label{maxcond}%
\end{equation}
}

\begin{example}
{\normalsize Suppose that the following condition holds: }
\end{example}

{\normalsize $g$ is independent of $\mathbf{x}$ and $m$, for example let us
take\
\[
g(t,\pi)=-\alpha Y(t)+\beta\mathbb{E}[Y(t)]-\ln\pi.
\]
}

{\normalsize Then $(p,0,0)$ where $p$ solves the deterministic equation
\[
\hat{p}(t)=\hat{p}(T)-\int_{t}^{T}\frac{\partial}{\partial m}b_{0}(s,m)p(s)ds
\]
}

{\normalsize solves $(\ref{example:adEq1}),$for all finite\ $T$. Now by
$(\ref{maxcond})$ we deduce that
\begin{equation}
\hat{\pi}(t)=\frac{-\hat{\lambda}(t)}{\hat{p}(t)} \label{pihut}%
\end{equation}
}

{\normalsize with%
\begin{equation}
\frac{\partial}{\partial\pi}g(t,\pi)=\frac{-1}{\pi(t)}. \label{pi}%
\end{equation}
}

{\normalsize Consequently%
\begin{equation}
\lambda(t)=e^{-\left(  \alpha-\beta\right)  t}\text{ for all }t\in
\lbrack0,\infty). \label{lamda}%
\end{equation}
}

{\normalsize Combining $(\ref{pi})$ and\ $(\ref{maxcond}),$\ if $\pi$ is
bounded away from $0$ we have%
\begin{equation}
\hat{p}(T)=\frac{e^{-\left(  \alpha-\beta\right)  T}}{\hat{\pi}(T)}%
\underset{T\rightarrow\infty}{\rightarrow}0\text{ if }\beta<\alpha.
\label{phut}%
\end{equation}
}

{\normalsize Put $\pi=0$ in equation $(\ref{SDE}),$ integrating and taking
expectation, we obtain%
\[
h(t):=\mathbb{E}\left[  X(t)\right]  =x+\int_{0}^{t}b_{0}(s,\mathbb{E}%
[X(s)])ds.
\]
}

{\normalsize First we assume that $b_{0}$ has at most linear growth, in the
sense that there exists a constant $c$ such that $b_{0}(t,x)\leq cx.$ Then we
get%
\[
h(t)\leq x+c\int_{0}^{t}h(s)ds
\]
}

{\normalsize and hence by the Gronwall inequality it follows that
\begin{equation}
h(t)\leq xe^{ct}\text{ for all }t. \label{h}%
\end{equation}
}

{\normalsize For given consumption rate $\pi,$ let $X^{\pi}(t)$ be the
corresponding solution of $(\ref{SDE}).$ Then since $\pi(t)\geq0$ for all $t,$
we always have $X^{\pi}(t)$ $\leq X^{0}(t).$ Therefore, to prove the
transversality condition it suffices to prove that $\mathbb{E}\left[  \hat
{p}(T)X^{0}(T)\right]  $ goes to $0$ as $T$ goes to infinity. }

{\normalsize Let us compare $(\ref{h})$ with the decay of $\hat{p}%
(T\mathbf{)}$ in $(\ref{phut})$ we get}

{\normalsize
\begin{align}
\mathbb{E}\left[  \hat{p}(T)X^{0}(T)\right]   &  =\mathbb{E}\left[  \hat
{p}(T)X(T)\right] \nonumber\\
&  =\hat{p}(T)\mathbb{E}\left[  X(T)\right] \nonumber\\
&  =\frac{xe^{-\left(  \alpha-\beta-c\right)  T}}{\pi}\underset{T\rightarrow
\infty}{\rightarrow}0\text{ if }c<\left(  \alpha-\beta\right)  .
\end{align}
}

{\normalsize \include{InfiniteHorizonMFBSDE} }

\end{document}